\documentclass[12pt,a4paper]{article}
\usepackage{a4wide}
\usepackage[utf8]{inputenc}
\usepackage{enumerate}
\usepackage{amsmath}
\usepackage{amsfonts}
\usepackage{amssymb}
\usepackage{bbm}
\usepackage{mathrsfs}
\usepackage{amsthm}
\usepackage{verbatim}
\usepackage{upgreek}
\usepackage{tikz-cd}
\usepackage{relsize}
\usepackage{mathtools}
\usepackage{enumitem} 

\usepackage{dsfont} 
\usepackage{graphicx} 

\newcommand{\ov}{\overline}
\newcommand{\id}{\textnormal{id}}
\newcommand{\mc}{\mathcal}

\newcommand{\mf}{\mathfrak}
\newcommand{\msf}{\mathsf}
\newcommand{\I}{\mathbbm{1}}
\newcommand{\vp}{\varphi}
\newcommand{\Lvp}{\Lambda_{\vp}}
\newcommand{\Lh}{\Lambda_{h}}
\newcommand{\Lpsi}{\Lambda_{\psi}}
\newcommand{\Lhvps}{\Lambda_{\wh{\psi}}}
\newcommand{\md}{\operatorname{d}\!}
\newcommand{\cst}{\ifmmode \mathrm{C}^* \else $\mathrm{C}^*$\fi}

\newcommand{\NN}{\mathbb{N}}
\newcommand{\RR}{\mathbb{R}}

\newcommand{\CC}{\mathbb{C}}
\newcommand{\ZZ}{\mathbb{Z}}
\newcommand{\GG}{\mathbb{G}}

\newcommand{\TT}{\mathbb{T}}

\newcommand{\bbGamma}{{\mathpalette\makebbGamma\relax}}
\newcommand{\makebbGamma}[2]{%
  \raisebox{\depth}{\scalebox{1}[-1]{$\mathsurround=0pt#1\mathbb{L}$}}%
}

\newcommand{\IrrGamma}{\Irr(\bbGamma)}
\newcommand{\mrW}{\mathrm{W}}
\newcommand{\mrV}{\mathrm{V}}

\newcommand{\ismaa}[2]{\langle#1\,|\,#2\rangle}

\newcommand{\wot}{\ifmmode \textsc{wot} \else \textsc{wot}\fi}
\newcommand{\sot}{\ifmmode \textsc{sot} \else \textsc{sot}\fi}
\newcommand{\sots}{\ifmmode \textsc{sot}^* \else \textsc{sot}$^*$\fi}
\newcommand{\ssot}{\ifmmode \sigma\textsc{-sot} \else $\sigma$-\textsc{sot }\fi}
\newcommand{\ssots}{\ifmmode \sigma\textsc{-sot}^* \else $\sigma$-\textsc{sot }$^*$\fi}
\newcommand{\swot}{\ifmmode \sigma\textsc{-wot} \else $\sigma$-\textsc{wot}\fi}

\newcommand{\Linf}{\operatorname{L}^{\infty}(\GG)}
\newcommand{\Linfd}{\operatorname{L}^{\infty}(\whG)}
\newcommand{\Lj}{\operatorname{L}^{1}(\GG)}
\newcommand{\CG}{\mathrm{C}_0(\GG)}
\newcommand{\CGD}{\mathrm{C}_0(\whG)}
\newcommand{\CGDu}{\mathrm{C}_0^{u}(\whG)}

\newcommand{\wh}{\widehat}
\newcommand{\whG}{\widehat{\GG}}
\newcommand{\hvp}{\widehat{\vp}}
\newcommand{\Lhvp}{\Lambda_{\hvp}}

\newcommand{\LdG}{\operatorname{L}^{2}(\GG)}
\newcommand{\IrrG}{\Irr(\GG)}
\newcommand{\hpsi}{\widehat{\psi}}

\newcommand{\SUd}{\operatorname{SU}_q(2)}

\newcommand{\WW}{\text{\reflectbox{$\mathds{W}$}}\:\!} 

\DeclareMathOperator{\lin}{span}
\DeclareMathOperator{\Irr}{Irr}
\DeclareMathOperator{\Pol}{Pol}
\DeclareMathOperator{\Tr}{Tr}
\DeclareMathOperator{\B}{B}
\DeclareMathOperator{\M}{M}
\DeclareMathOperator{\Dom}{Dom}

\DeclareMathOperator{\LL}{L}
\DeclareMathOperator{\HS}{HS}
\DeclareMathOperator{\Diag}{Diag}
\DeclareMathOperator{\diag}{diag}

\DeclareMathOperator{\sgn}{sgn}

\newtheorem{theorem}{Theorem}[section]
\newtheorem{proposition}[theorem]{Proposition}
\newtheorem{lemma}[theorem]{Lemma}
\theoremstyle{definition}
\newtheorem{corollary}[theorem]{Corollary}
\newtheorem{remark}[theorem]{Remark}
\newtheorem{definition}[theorem]{Definition}
\numberwithin{equation}{section}

\title{Modular properties of type I locally compact quantum groups}
\author{Jacek Krajczok\thanks{Email address: jkrajczok@impan.pl}\\ Institute of Mathematics, Polish Academy of Sciences}
\date{}

\begin{document}
\maketitle

\begin{abstract}
The following paper is devoted to the study of type I locally compact quantum groups. We show how various operators related to the modular theory of the Haar integrals on $\GG$ and $\whG$ act on the level of direct integrals. Using these results we derive a web of implications between properties such as unimodularity or traciality of the Haar integrals. We also study in detail two examples: discrete quantum group $\wh{\SUd}$ and the quantum $az+b$ group.
\end{abstract}

\section{Introduction}
A remarkable feature of the theory of compact quantum groups introduced by Woronowicz (\cite{Woronowiczsu2, cqg}) is the fact that the Haar integral need not be tracial (in such case one says that a compact quantum group $\GG$ is not of Kac type). Whether $\GG$ is of Kac type or not, is related to a number of other properties. To name a few, the Haar integral of $\GG$ is tracial if, and only if its scaling group is trivial and this happens if, and only if the dual discrete quantum group $\whG$ is unimodular (equivalently has tracial integrals). In fact, behind all these objects and properties stands a family $(\uprho_\alpha)_{\alpha\in\IrrG}$ of positive invertible operators (see \cite{NeshTu}) and $\GG$ is of Kac type if, and only if $\uprho_\alpha=\I_{\msf{H}_\alpha}$ for all $\alpha\in\IrrG$.\\
A theory of locally compact quantum groups was proposed by Kustermans and Vaes (\cite{KustermansVaes}). As general quantum group can be non-unimodular, each quantum group $\GG$ has two Haar integrals: left $\vp$ and right $\psi$. It is still possible that these are non-tracial, however now the situation is more complicated and the above simple equivalences from the world of compact quantum groups are no longer valid.\\
An intermediate step between the theory of compact and general locally compact quantum groups is formed by the so called type I locally compact quantum groups. Roughly speaking, similarly to the classical (i.e.~not quantum) setting, these are quantum groups with type I universal group \cst-algebras. Their study was initiated in the doctoral dissertation of Desmedt \cite{Desmedt} where he has constructed a Plancherel measure on $\IrrG$ and described its properties. Together with the Plancherel measure come two fields of strictly positive selfadjoint operators, $(D_\pi)_{\pi\in\IrrG}$ and $(E_\pi)_{\pi\in\IrrG}$. They can be thought of as replacements of the operators $\uprho_\alpha$ from the compact theory; for example, the Haar integrals of $\whG$ are tracial if, and only if almost all operators $D_\pi,E_\pi$ are multiples of the identity. One of the main results of our paper is a theorem which describes a relation between various properties of $\GG$ and $\whG$ (unimodularity, traciality of the Haar integrals, trivial scaling group etc.) and properties of operators $D_\pi,\,E_\pi\, (\pi\in\IrrG)$ -- this is accomplished in Section \ref{secspecial}.\\
In the next section we introduce objects used in the paper and set up the notation. Section \ref{secpreliminaries} is devoted to introducing a notion of matrix coefficients (in type I case) and recalling results of Desmedt (\cite{Desmedt}) and Caspers (\cite{Caspers}) which are used later in the paper. In Section \ref{secrel} we describe the polar decomposition of the map $T'\colon \Lambda_\psi(x)\mapsto\Lvp(x^*)$ coming from the Tomita-Takesaki theory and as a corollary we get an important relation between unitary operators $\mc{Q}_L,\mc{Q}_R$ from the Desmedt's theorem. In Section \ref{secoperators} we show how various operators act on the level of direct integrals. We remark that a formula for $\nabla_{\hvp}^{it}$ from Theorem \ref{tw1} was recently used in \cite{quantumdisk} to deduce that the Toeplitz algebra is not an algebra of continuous functions on a compact quantum group. Finally, in Section \ref{secexamples} we describe two interesting examples of type I locally compact quantum groups: discrete group $\widehat{\SUd}$ and the quantum $"az+b"$ group.

\section{Notation}\label{secnotation}
Throughout the paper, $\GG$ will be a \textbf{locally compact quantum group} in the sense of Kustermans and Vaes. W refer the reader to papers \cite{KustermansVaes, Daele} for an introduction to the subject, here we will recall only necessary facts. Quantum group $\GG$ comes together with a number of objects: first of all we have a von Neumann algebra $\Linf$, a normal unital $\star$-homomorphism $\Delta_{\GG}\colon \Linf\rightarrow \Linf\bar\otimes\Linf$ called \textbf{comultiplication} and two n.s.f.~weights on $\Linf$: $\vp$ and $\psi$. They are called respectively the left and the right \textbf{Haar integral} as they satisfy certain invariance conditions. We will write $\Lvp,\Lpsi$ for the GNS maps. The GNS Hilbert spaces $\msf{H}_\vp,\msf{H}_\psi$ can be identified and will be denoted by $\LdG$. We will write $(\sigma^{\vp}_t)_{t\in\RR},\nabla_\vp,J_\vp$ for the group of \textbf{modular automorphisms} associated with the weight $\vp$, the \textbf{modular operator} and the \textbf{modular conjugation} -- an analogous notation will be used for other weights. The predual of $\Linf$ will be denoted by $\Lj$.
With every locally compact quantum group $\GG$ one can associate its \textbf{dual} $\whG$. The objects associated with $\whG$ will be decorated with hats. The Hilbert spaces $\LdG$, $\LL^2(\whG)$ can (and will) be identified. We will use a \cst-algebra $\CG\subseteq\B(\LdG)$. It is a $\swot$ dense subalgebra of $\Linf$. 
An important role in the theory plays the \textbf{Kac-Takesaki operator}\footnote{Symbol $\otimes$ stands for the minimal tensor product of \cst-algebras or the tensor product of Hilbert spaces.} $\mrW\in \M(\CG\otimes\CGD)$. It is a unitary operator characterized by the property $((\omega\otimes\id)\mrW^* )\Lvp(x)=\Lvp((\omega\otimes\id)\Delta_{\GG}(x))\;(\omega\in\Lj,x\in\mf{N}_{\vp})$. We note that the right leg of $\mrW$ generates $\CGD$ -- this means that the map $\lambda\colon \Lj \ni \omega\mapsto (\omega\otimes\id)\mrW\in \CGD$ satisfies $\ov{\lambda(\Lj)}=\CGD$. There is also a unitary $\mrV\in \Linfd'\bar\otimes \Linf$ related to the right Haar integral $\psi$. With $\GG$ one can associate yet another \cst-algebra, $\mathrm{C}_0^{u}(\GG)$ called \textbf{the universal version} of $\CG$. It is related to $\CG$ via so called \textbf{reducing morphism} $\Lambda_{\GG}\colon \mathrm{C}_0^{u}(\GG)\rightarrow \CG$. We remark that $\CGDu$ plays a role of the full group \cst-algebra and its representations are in bijection with unitary representations of $\GG$ (see \cite{Kustermans, MUQGII}). To be more precise, there exists a unitary operator $\WW\in\M(\CG\otimes \CGDu)$ such that every unitary representation of $\GG$ on a Hilbert space $\msf{H}_\pi$ is of the form $U^{\pi}=(\id\otimes\pi)\WW$ for a nondegenerate representation $\pi\colon\CGDu\rightarrow \B(\msf{H}_\pi)$. This correspondence preserves irreducibility -- consequently the spectrum of $\CGDu$ will be denoted by $\IrrG$.\\
Besides the groups of modular automorphisms $(\sigma^{\vp}_t)_{t\in\RR}, (\sigma^{\psi}_t)_{t\in\RR}$ there is also a third group of automorphisms of $\Linf$ -- $(\tau_t)_{t\in\RR}$, called the \textbf{scaling group}. It is implemented by a strictly positive selfadjoint operator $P$: $\tau_t(x)=P^{it} x P^{-it}\,(x\in \Linf,t\in\RR)$. We remark that $P$ is selfdual: we have $\hat{P}=P$. The Haar integrals are relatively invariant under the scaling group: we have $\vp\circ \tau_t=\nu^{-t} \vp,\psi\circ \tau_t=\nu^{-t}\psi\,(t\in\RR)$ for a number $\nu>0$ called the \textbf{scaling constant}. The scaling constant relates the modular conjugations for $\vp$ and $\psi$: we have $J_\psi=\nu^{\frac{i}{4}} J_\vp$. An important role in our paper is played by the so called \textbf{modular element} $\delta$. It is a strictly positive selfadjoint operator affiliated with $\Linf$ which apppears as (a part of) the Radon--Nikodym derivative between $\psi$ and $\vp$. There is a plethora of formulas which relates the above objects. Let us end this part of the introduction with a collection of them -- we will use it a lot throught the paper: 
\begin{equation}\begin{split}\label{eq20}
J_{\hvp} J_\vp &= \nu^{\frac{i}{4}} J_\vp J_{\hvp},\quad
\nabla_\psi^{it}=J_{\hvp} \nabla_\vp^{-it} J_{\hvp},\quad
J_{\hvp} \delta^{it}=\delta^{it} J_{\hvp},\quad
\nabla_\psi^{it}=\hat{\delta}^{-it} P^{-it}\\
\nabla_\vp^{is}\delta^{it}&= \nu^{ist} \delta^{it}\nabla_\vp^{is},\quad
\nabla_\psi^{is}\delta^{it}= \nu^{ist} \delta^{it}\nabla_\psi^{is},
\quad J_\vp P^{it} = P^{it} J_\vp,\quad
P^{it} \delta^{is}=\delta^{is} P^{it}\\
P^{it}\nabla_\vp^{is}&=\nabla_\vp^{is} P^{it},\quad
P^{it}\nabla_\psi^{is}=\nabla_\psi^{is} P^{it},
\end{split}\end{equation}
where $t$ and $s$ are arbitrary reals numbers. The above properties belong to the standard theory of locally compact quantum groups, their proofs can be found in \cite{Kustermans, KustermansVaes, Daele}.\\
Throught the paper, we will extensively use the theory of direct integrals -- we refer the reader to \cite{DixmiervNA, Lance} for basic notions and properties. Let us mention here only that if $(\msf{H}_x)_{x\in X}$ is a measurable field o Hilbert spaces, then $\int_X^{\oplus} \msf{H}_x\md\mu(x)$ is a Hilbert space which consists of (classes of) measurable vector fields\footnote{This means that $(\msf{H}_x)_{x\in X}$ comes together with a choice of a fundamental sequence $\{(\xi^n_x)_{x\in X}\,|\, n\in\NN\}$ and for each $n\in\NN$, the function $X\ni x \mapsto \ismaa{\xi^n_x}{\xi_x}\in \CC$ is measurable. We will neglect mentioning the fundamental sequence and simply say that $(\msf{H}_x)_{x\in X}$ is a measurable field of Hilbert spaces.} $\xi=(\xi_x)_{x\in X}$ satisfying $\int_X \|\xi_x\|^2\md\mu(x)<+\infty$. Usually one also writes $\xi=\int_X^{\oplus} \xi_x\md\mu(x)$. For two closed operators $A$ and $B$, the symbol $A\circ B$ will stand for the operator given by $A\circ B(\xi)=A(B(\xi))$ on the domain $\Dom(A\circ B)=\{\xi\in \Dom(B)\,|\, B\xi\in \Dom(A)\}$. Whenever $A\circ B$ is closable, we will denote its closure by $AB$. The complex conjugate of a Hilbert space $\msf{H}$ will be denoted by $\ov{\msf{H}}$. For an operator $A$ on $\msf{H}$, $A^{\msf T}$ will be an operator on $\ov{\msf{H}}$ given by $A^{\msf T}\, \ov{\xi}=\ov{A^* \xi}\,(\xi\in\msf{H})$. If $\pi$ is a representation of $\GG$ on $\msf{H}_\pi$ then we associate with it a representation $\pi^c={\cdot}^\msf{T}\circ \pi\circ\wh{R}^u$ on $\ov{\msf{H}_\pi}$, where $\wh{R}^u$ is the unitary antipode on $\CGDu$ (\cite{Kustermans, MUQGII}).
All scalar products are linear on the right.

\section{Preliminaries}\label{secpreliminaries}
Let us introduce two notions: we say that $\GG$ is \textbf{second countable}\footnote{This condition is equivalent to number of other separability assumptions, see \cite[Lemma 14.6]{Krajczok}. We note that these conditions are satisfied for $\GG$ if and only they are satisfied for $\whG$.} if $\CGDu$ separable and \textbf{type I} if $\CGDu$ is of type I.
Our work is based on the work of Desmedt \cite{Desmedt} and Caspers and Koelink \cite{Caspers,CaspersKoelink}. First of all, we will use the fundamental result of Desmedt which states existence of the Plancherel measure and its properties (see also \cite{Krajczok} and discussion therein). We recall only parts that will be used in the paper.

\begin{theorem}\label{PlancherelL}
Let $\GG$ be a second countable, type I locally compact quantum group. There exists a standard measure $\mu$ on $\IrrG$, a measurable field of Hilbert spaces $(\msf{H}_\pi)_{\pi\in \IrrG}$, measurable field of representations\footnote{We use the same symbol $\pi$ for a class of representations and its representative chosen according to a fixed measurable field of representations.}, measurable field of strictly positive self-adjoint operators $(D_\pi)_{\pi\in \IrrG}$ and a unitary operator $\mc{Q}_L\colon \LdG\rightarrow \int_{\IrrG}^{\oplus} \HS(\msf{H}_\pi)\md\mu(\pi)$ such that:
\begin{enumerate}[label=\arabic*)]
\item For all $\alpha\in \Lj$ such that $\lambda(\alpha)\in\mf{N}_{\hvp}$ and $\mu$-almost every $\pi\in\IrrG$ the operator $(\alpha\otimes\id) (U^{\pi}) \circ D_{\pi}^{-1}$ is bounded and its closure $(\alpha\otimes\id) (U^{\pi}) D_{\pi}^{-1}$ is Hilbert-Schmidt.
\item The operator $\mc{Q}_L$ is the isometric extension of
\[
\Lhvp(\lambda(\Lj)\cap \mf{N}_{\hvp})\ni \Lhvp(\lambda(\alpha))\mapsto
\int_{\IrrG}^{\oplus} 
(\alpha\otimes\id) (U^{\pi}) D_{\pi}^{-1}\md\mu(\pi)
\in
\int_{\IrrG}^{\oplus} \HS(\msf{H}_\pi)\md\mu(\pi),
\]
\item
The operator $\mc{Q}_L$ satisfies the following equations:
\[
\mc{Q}_L (\omega\otimes\id)\mrW=
\bigl(\int_{\IrrG}^{\oplus} (\omega\otimes\id)U^{\pi}\otimes\I_{\ov{\msf{H}_{\pi}}}\md\mu(\pi)\bigr)\mc{Q}_L
\]
and
\[
\mc{Q}_L (\omega\otimes\id)\chi(\mrV)=
\bigl(\int_{\IrrG}^{\oplus} \I_{\msf{H}_\pi}\otimes \pi^{c}((\omega\otimes\id){\WW})\md\mu(\pi)\bigr)\mc{Q}_L
\]
for every $\omega\in\LL^1(\GG)$.
\item
Haar integrals on $\whG$ are tracial if and only if almost all $D_{\pi}$ are multiples of the identity.
\item The operator $\mc{Q}_L$ transforms $\LL^{\infty}(\whG)\cap \LL^{\infty}(\whG)'$ into diagonalisable operators.
\item We can assume that $(\msf{H}_\pi)_{\pi\in\IrrG}$ is the canonical measurable field of Hilbert spaces.
\end{enumerate}
\end{theorem}

We have also the right version of the above theorem.

\begin{theorem}\label{PlancherelR}
Let $\GG$ be a second countable, type I locally compact quantum group. There exists a standard measure $\mu^R$ on $\IrrG$ a measurable field of Hilbert space $(\msf{K}_\pi)_{\pi\in \IrrG}$, measurable field of representations, measurable field of strictly positive self-adjoint operators $(E_\pi)_{\pi\in \IrrG}$ and a unitary operator $\mc{Q}_R\colon \LdG\rightarrow \int_{\IrrG}^{\oplus} \HS(\msf{K}_\pi)\md\mu^{R}(\pi)$ such that:
\begin{enumerate}[label=\arabic*)]
\item For all $\alpha\in \LL^1(\GG)$ such that $\lambda(\alpha)\in\mf{N}_{\hpsi}$ and $\mu^R$-almost every $\pi\in\IrrG$ the operator $(\alpha\otimes\id) (U^{\pi}) \circ E_{\pi}^{-1}$ is bounded and its closure $(\alpha\otimes\id) (U^{\pi}) E_{\pi}^{-1}$ is Hilbert-Schmidt.
\item The operator $\mc{Q}_R$ is the isometric extension of
\[
\begin{split}
J_{\hvp} J_\vp \Lhvps(\lambda(\LL^1(\GG))\cap \mf{N}_{\hpsi})\ni &
J_{\hvp} J_\vp \Lhvps(\lambda(\alpha))\mapsto\\
&\mapsto
\int_{\IrrG}^{\oplus} 
(\alpha\otimes\id) (U^{\pi}) E_{\pi}^{-1}\md\mu^R(\pi)
\in
\int_{\IrrG}^{\oplus} \HS(\msf{K}_\pi)\md\mu^R(\pi),
\end{split}
\]
\item
The operator $\mc{Q}_R$ satisfies the following equations:
\[
\mc{Q}_RJ_{\hvp}J_\vp (\omega\otimes\id)\mrW=
\bigl(\int_{\IrrG}^{\oplus} (\omega\otimes\id)U^{\pi}\otimes\I_{\ov{\msf{H}_{\pi}}}\md\mu^R(\pi)\bigr)\mc{Q}_R J_{\hvp} J_\vp
\]
and
\[
\mc{Q}_R J_{\hvp} J_\vp(\omega\otimes\id)\chi(\mrV)=
\bigl(\int_{\IrrG}^{\oplus} \I_{\msf{H}_\pi}\otimes \pi^{c}((\omega\otimes\id){\WW})\md\mu^R(\pi)\bigr)\mc{Q}_RJ_{\hvp} J_\vp
\]
for every $\omega\in\LL^1(\GG)$.
\item
Haar inegrals on $\whG$ are tracial if and only if almost all $E_{\pi}$ are multiples of the identity.
\item The operator $\mc{Q}_R$ transforms $\LL^{\infty}(\whG)\cap \LL^{\infty}(\whG)'$ into diagonalisable operators.
\item
We can choose $\mu^R=\mu$ and $\msf{K}_\pi=\msf{H}_\pi$ (and the same field of representations as in Theorem \ref{PlancherelL}).
\end{enumerate}
\end{theorem}

From now on, let $\GG$ be a second countable, type I locally compact quantum group and choose all the objects provided by theorems \ref{PlancherelL}, \ref{PlancherelR}. The last point of the above theorem allows us to assume $\mu^R=\mu$ and $\msf{K}_\pi=\msf{H}_\pi$. Let us introduce two strictly positive, selfadjoint operators $D=\int_{\IrrG}^{\oplus} D_\pi \md\mu(\pi)$ and $E=\int_{\IrrG}^{\oplus} E_\pi \md\mu(\pi)$. We will use plenty of times the following easily derived property:
\begin{proposition}\label{stw1}
Define an antiunitary operator $\Sigma=\int_{\IrrG}^{\oplus} J_{\msf{H}_\pi} \md\mu(\pi)$, where
\[
J_{\msf{H}_\pi}\colon \msf{H}_\pi\otimes\ov{\msf{H}_\pi}\ni
\xi\otimes\ov{\eta}\mapsto\eta\otimes\ov\xi\in
\msf{H}_\pi\otimes\ov{\msf{H}_\pi}\quad(\pi\in\IrrG).
\]
We have
\[
\nu^{\frac{i}{4}}J_{\hpsi}=J_{\hvp}=\mc{Q}_L^* \Sigma\mc{Q}_L= \mc{Q}_R^*\Sigma\mc{Q}_R.
\]
\end{proposition}

\begin{proof}
Let $\hvp^u$ be the left Haar weight on the universal \cst-algebra $\CGDu$. Its GNS construction is $(\LdG,\Lambda_{\whG},\Lambda_{\hvp}\circ\Lambda_{\whG})$ (see \cite{Kustermans}), hence $J_{\hvp}$ is the modular conjugation for $\hvp^{u}$. It is transformed to $\Sigma$ by $\mc{Q}_L$ -- it is a part of the Desmedt's result.\\
Similarly, $\mc{Q}_{R,0}$ transforms $J_{\hpsi}$ to $\Sigma$: $J_{\hpsi}=\mc{Q}_{R,0}^* \Sigma \mc{Q}_{R,0}$. Operator $\mc{Q}_R$ is defined as $\mc{Q}_R=\mc{Q}_{R,0} J_\vp J_{\hvp}$ (see \cite[Theorem 3.4]{Krajczok}). Consequently, we get $J_{\hpsi}= J_{\vp} J_{\hvp} \mc{Q}_R^*\Sigma \mc{Q}_R J_{\hvp} J_{\vp}$. Using the commutation relation $J_{\hvp} J_\vp =\nu^{\frac{i}{4}} J_\vp J_{\hvp}$ (see equation \eqref{eq20}) and formula $J_{\hpsi}=\nu^{-\frac{i}{4}} J_{\hvp}$ (the scaling constant of $\whG$ is $\hat{\nu}=\nu^{-1}$) we arrive at
\[
\mc{Q}_R^* \Sigma \mc{Q}_R=
J_{\hvp} J_{\vp}(\nu^{-\frac{i}{4}}J_{\hvp}) J_{\vp} J_{\hvp}=
\nu^{-\frac{i}{4}} \nu^{\frac{i}{4}} J_\vp J_{\hvp}J_{\hvp}  J_\vp J_{\hvp}=J_{\hvp}.
\]
\end{proof}

Let us note in the next proposition how $\mc{Q}_L,\mc{Q}_R$ transform $\LL^{\infty}(\whG)$ and its comutant.

\begin{proposition}\label{stw7}
We have the following equalities of von Neumann algebras:
\[\begin{split}
\mc{Q}_L \LL^{\infty}(\whG) \mc{Q}_L^*=
\int_{\IrrG}^{\oplus} \B(\msf{H}_\pi)\otimes \I_{\ov{\msf{H}_\pi}} 
\md\mu(\pi) &,\quad
\mc{Q}_L \LL^{\infty}(\whG)' \mc{Q}_L^*=
\int_{\IrrG}^{\oplus} \I_{\msf{H}_\pi}\otimes \B(\ov{\msf{H}_\pi})
\md\mu(\pi)\\
\mc{Q}_R \LL^{\infty}(\whG) \mc{Q}_R^*=
\int_{\IrrG}^{\oplus} \I_{\msf{H}_\pi}\otimes \B(\ov{\msf{H}_\pi})
\md\mu(\pi)
&,\quad
\mc{Q}_R \LL^{\infty}(\whG)' \mc{Q}_R^*=
\int_{\IrrG}^{\oplus} \B(\msf{H}_\pi)\otimes \I_{\ov{\msf{H}_\pi}} \md\mu(\pi)
 \end{split}\]
\end{proposition}

The first part of the above result is a result of Desmedt. The second one can be derived as in the proof of Proposition \ref{stw1}, using equation $\mc{Q}_R=\mc{Q}_{R,0} J_{\vp} J_{\hvp}$. Let us now define analogs of the matrix coefficients $U^{\alpha}_{i,j}$ used in the theory of compact quantum groups. Elements of this form were already considered in \cite{Caspers}.

\begin{definition}
For $\xi,\eta\in\int_{\IrrG}^{\oplus}\msf{H}_\pi\md\mu(\pi)$ we define elements of $\Linf$:
\[
M^L_{\xi,\eta}=\int_{\IrrG} (\id\otimes\omega_{\xi_\pi,\eta_\pi}) (U^{\pi *}) \md\mu(\pi),\quad
M^R_{\xi,\eta}=\int_{\IrrG} (\id\otimes\omega_{\xi_\pi,\eta_\pi}) (U^{\pi }) \md\mu(\pi).
\]
The above elements will be referred to as left (resp.~right) matrix coefficients.
\end{definition}

Note that the above (weak) integrals converge in $\swot$ and we have $(M^L_{\xi,\eta})^*=M^R_{\eta,\xi}$.\\

Our further reasoning is based on results derived by Caspers and Koelink in \cite{Caspers,CaspersKoelink}. We remark that one needs to be careful when taking equations from these papers as there is a difference in convention: we prefer to use inner products linear on the right and functionals $\omega_{\xi,\eta}$ defined accordingly. That is why we choose to state explicitly used results with necessary changes, which we do in this section. \\
First, we can transport a left (resp.~right) matrix coefficient via $\mc{Q}_L$ (resp.~$\mc{Q}_R$). The following is a reformulation of \cite[Lemma 3.7, Lemma 3.9]{CaspersKoelink}.
\newpage
\begin{lemma}\label{lemat5}$ $
\begin{enumerate}[label=\arabic*)]
\item If $\xi,\eta\in\int_{\IrrG}^{\oplus}\msf{H}_\pi\md\mu(\pi)$, $\xi\in \Dom(D)$ and the vector field $(\eta_\pi\otimes\ov{D_\pi\xi_\pi})_{\pi\in\IrrG}$ is square integrable, then $M^L_{\xi,\eta}\in\mf{N}_{\vp}$ and
$\mc{Q}_L\Lambda_{\vp}(M^L_{\xi,\eta})=
\int_{\IrrG}^{\oplus}\eta_\pi\otimes \ov{D_\pi\xi_\pi}\md\mu(\pi)$.
\item
If $\xi,\eta\in\int_{\IrrG}^{\oplus}\msf{H}_\pi\md\mu(\pi)$, $\xi\in \Dom(E)$ and the vector field $(\eta_\pi\otimes\ov{E_\pi\xi_\pi})_{\pi\in\IrrG}$ is square integrable, then $M^R_{\xi,\eta}\in\mf{N}_{\psi}$ and $\mc{Q}_R\Lambda_{\psi}(M^R_{\xi,\eta})=
\int_{\IrrG}^{\oplus}\eta_\pi\otimes \ov{E_\pi\xi_\pi}\md\mu(\pi)$.
\end{enumerate}
\end{lemma}

Using the above result and the fact that $\mc{Q}_L,\mc{Q}_R$ are unitary, one can easily derive the following density results:

\begin{lemma}\label{lemat13}
$ $
\begin{enumerate}[label=\arabic*)]
\item Set $\{\Lvp(M^L_{\xi,\eta})\},$ where $\xi,\eta$ run over vectors in $\int_{\IrrG}^{\oplus}\msf{H}_\pi \md\mu(\pi)$ such that $\xi\in\Dom(D)$ and $(\eta_\pi\otimes \ov{D_\pi\xi_\pi})_{\pi\in\IrrG}$ is square integrable, is lineary dense in $\LdG$.
\item Set $\{\Lambda_{\psi}(M^R_{\xi,\eta})\}$, where $\xi,\eta$ run over vectors in $\int_{\IrrG}^{\oplus}\msf{H}_\pi \md\mu(\pi)$ such that $\xi\in\Dom(E)$ and $(\eta_\pi\otimes \ov{E_\pi\xi_\pi})_{\pi\in\IrrG}$ is square integrable, is lineary dense in $\LdG$.
\end{enumerate}
\end{lemma}

Consider an antilinear map\footnote{This map appears during a construction of the Radon-Nikodym derivative between $\psi$ and $\vp$, see \cite{TakesakiII}.}
\begin{equation}\label{eq4}
\Lambda_\psi( \mf{N}_\psi \cap {\mf{N}_\vp}^*)\ni \Lambda_\psi(x)
\mapsto
\Lvp(x^*)\in\LdG
\end{equation}
and define $T'$ to be its closure. Let $T'=J'{\nabla'}^{\frac{1}{2}}$ be the polar decomposition of $T'$. It is well known that $J'$ is antiunitary and ${\nabla'}^{\frac{1}{2}}$ is strictly positive and selfadjoint. In the next section we will describe these operators, for now let us recall how they look on the level of direct integrals.

\begin{proposition}\label{stw5}
We have $\mc{Q}_L J' \mc{Q}_R^*=\Sigma$ and $\mc{Q}_R {\nabla'}^{\frac{1}{2}}\mc{Q}_R^*=\int_{\IrrG}^{\oplus} D_\pi \otimes (E_\pi^{-1})^{\msf T} \md\mu(\pi)$.
\end{proposition}

The above proposition is a combination of \cite[Proposition 4.4, Proposition 4.5, Theorem 4.6]{CaspersKoelink}. We finish this section with formulas expressing the action of modular automorphism groups on the matrix coefficients.

\begin{proposition}\label{stw4}
For each $\xi,\eta\in \int_{\IrrG}^{\oplus}\msf{H}_\pi \md\mu(\pi),t\in\RR$ the following holds:
\[
\begin{split}
\sigma^{\psi}_t(M^R_{\xi,\eta})=\nu^{\frac{1}{2}it^2}
\delta^{it}\,M^R_{E^{2it}\xi,D^{2it}\eta}&,\quad
\sigma^{\vp}_t(M^R_{\xi,\eta})=\nu^{\frac{1}{2}it^2}
\,M^R_{E^{2it}\xi,D^{2it}\eta}\, \delta^{it},\\
\sigma^{\psi}_t(M^L_{\xi,\eta})=\nu^{-\frac{1}{2}it^2}
M^L_{D^{2it}\xi,E^{2it}\eta}\, \delta^{-it}&,\quad
\sigma^{\vp}_t(M^L_{\xi,\eta})=\nu^{-\frac{1}{2}it^2}
 \delta^{-it}\,M^L_{D^{2it}\xi,E^{2it}\eta}.
\end{split}
\]
\end{proposition}

The formulas expressing the action of $\sigma^\vp,\sigma^\psi$ on $M^R_{\xi,\eta}$ are stated in \cite[Remark 2.2.11]{Caspers}. The other two follow by taking the adjoint. We note that they can be derived using the formula for $\nabla'$ (Proposition \ref{stw5}) and equation $\nu^{\frac{1}{2} it^2} \delta^{it}=\nabla_\psi^{it}\,{\nabla'}^{-it}$ (see \cite[Equations (29), (30), page 112]{TakesakiII} and the proof of \cite[Theorem 3.11]{Daele}).

\section{Relation between $\mc{Q}_L$ and $\mc{Q}_R$}\label{secrel}
In this section we will describe the polar decompostion of the closed operator $T'\colon \Lpsi(x)\mapsto\Lvp(x^*)$ (see equation \eqref{eq4}), namely we will derive a equation $T'=(\nu^{\frac{i}{8}} J_\vp)(J_\vp \nu^{\frac{i}{8}} \nabla_\vp^{-\frac{1}{2}} \delta^{-\frac{1}{2}} J_\vp)$. As a corollary we get an important relation between $\mc{Q}_L$ and $\mc{Q}_R$. Before we do that, let us justify through a formal calculation, why the above formula for $T'$ should hold:
\begin{equation}\label{eq14}\begin{split}
&\quad\;
T'\Lambda_\psi(x)=\Lvp(x^*)=J_\vp \nabla_\vp^{\frac{1}{2}}\Lambda_\vp(x)=
J_\vp \nabla_\vp^{\frac{1}{2}} J_\vp \sigma^{\vp}_{i/2}(\delta^{-\frac{1}{2}})^* J_\vp \Lvp(x\delta^{\frac{1}{2}})\\
&=
\nabla_\vp^{-\frac{1}{2}} (\nu^{-\frac{i}{4}} \delta^{-\frac{1}{2}})^*
J_\vp\Lambda_\psi (x)=
(\nu^{\frac{i}{8}} J_\vp)(J_\vp \nu^{\frac{i}{8}} \nabla_\vp^{-\frac{1}{2}} \delta^{-\frac{1}{2}} J_\vp)\Lambda_\psi(x).
\end{split}\end{equation}
We need to include the factor $\nu^{\frac{i}{8}}$ due to the following lemma:

\begin{lemma}\label{lemat4}
For all $s,t\in \RR$ operators $\nabla_\vp^s\circ\delta^t,\,\delta^t\circ\nabla_\vp^s$ are closable. We have equality $\nu^{\frac{ist}{2}} \nabla_\vp^s \delta^t=
\nu^{-\frac{ist}{2}}
\delta^t \nabla_\vp^s$ of strictly positive, selfadjoint operators, moreover
\[
(\nu^{\frac{ist}{2}} \nabla_\vp^s \delta^t)^{ir}=
\nu^{-\frac{ist}{2} r^2} \nabla_\vp^{isr} \delta^{itr}=
\nu^{\frac{ist}{2} r^2} \delta^{itr}\nabla_\vp^{isr}\quad(r\in\RR).
\]
\end{lemma}

The above result is a consequence of the commutation relation $\nabla_\vp^{is}\delta^{it}=\nu^{ist} \delta^{it}\nabla_\vp^{is}\,(s,t\in\RR)$. Indeed, it follows that operators $\nabla_\vp^s,\delta^t$ satisfy the Weyl relation. Then Lemma \ref{lemat4} follows from \cite[Example 3.1, Theorem 3.1]{QEF}. The next lemma describes the action of the unbounded operator $\delta^t$.

\begin{lemma}\label{lemat17}$ $
\begin{enumerate}[label=\arabic*)]
\item Let $t\in\RR,x\in \mf{N}_\vp$ be such that $x\circ\delta^t$ is closable and $x\delta^t\in \mf{N}_\vp$. Then $J_\vp\Lvp(x)\in \Dom(\delta^t)$ and $\nu^{\frac{it}{2}}J_\vp \delta^t J_\vp \Lvp(x)=\Lvp(x\delta^t)$.
\item Let $t\in\RR,x\in \mf{N}_\psi$ be such that $x\circ\delta^t$ is closable and $x\delta^t\in \mf{N}_\psi$. Then $J_\vp\Lambda_\psi(x)\in \Dom(\delta^t)$ and $\nu^{\frac{it}{2}}J_\vp \delta^t J_\vp \Lambda_\psi(x)=\Lambda_\psi(x\delta^t)$.
\end{enumerate}
\end{lemma}

\begin{proof}
We prove only the first assertion, the second one can be derived analogously. Take $x\in\mf{N}_\vp,t\in\RR$ which satisfy conditions of the lemma and define
\[
x_n=\sqrt{\tfrac{n}{\pi}} \int_{\RR} e^{-np^2}x \delta^{ip} \md p\in\Linf\quad(n\in\NN)
\]
(the above weak integral converges in \swot). Operator $x_n\circ\delta^t$ is closable and we have
\begin{equation}\label{eq11}
x_n\delta^t=\sqrt{\tfrac{n}{\pi}} \int_{\RR} e^{-np^2}
(x\delta^t) \delta^{ip} \md p=
\sqrt{\tfrac{n}{\pi}} \int_{\RR} e^{-n(p+it)^2}
x \delta^{ip} \md p.
\end{equation}
Clearly $x_n,x_n\delta^t\in\mf{N}_{\vp}$ and due to the Hille's theorem
\[\begin{split}
&\quad\;
\Lvp(x_n)=\sqrt{\tfrac{n}{\pi}} \int_{\RR}e^{-np^2} \Lvp(x \delta^{ip}) \md p=
\sqrt{\tfrac{n}{\pi}} J_{\vp}\int_{\RR}e^{-np^2} \nu^{-\frac{p}{2}}\delta^{-ip}J_{\vp}\Lvp(x ) \md p,
\end{split}\]
similarly thanks to the equation \eqref{eq11} we have
\[
\Lvp(x_n\delta^t)=
\sqrt{\tfrac{n}{\pi}} J_{\vp}\int_{\RR}e^{-np^2} \nu^{-\frac{p}{2}}\delta^{-ip}J_{\vp}\Lvp(x\delta^t ) \md p=
\sqrt{\tfrac{n}{\pi}} J_{\vp}\int_{\RR}e^{-n(p-it)^2} \nu^{-\frac{p}{2}}\delta^{-ip}J_{\vp}\Lvp(x ) \md p.
\]
Consequently, $\Lvp(x_n)\xrightarrow[n\to\infty]{}\Lvp(x)$ and $\Lvp(x_n \delta^t)\xrightarrow[n\to\infty]{} \Lvp(x\delta^t)$. For each $r\in\RR$ we have
\[\begin{split}
&\quad\;
\delta^{ir} J_{\vp}\Lvp(x_n)=
\sqrt{\tfrac{n}{\pi}} \int_{\RR}e^{-np^2} 
\nu^{-\frac{p}{2}}\delta^{-i(p-r)}
J_{\vp}\Lvp(x ) \md p=f_n(r),
\end{split}\]
where $f_n$ is an entire function
\[
f_n\colon \CC\ni z \mapsto
\sqrt{\tfrac{n}{\pi}} \int_{\RR}e^{-n(p+z)^2} \nu^{-\frac{p+z}{2}}\delta^{-ip} J_{\vp}\Lvp(x)\md p\in\LdG.
\]
From the above follows that $J_\vp\Lvp(x_n)\in \Dom(\delta^z)$  for all $z\in \CC$ and $\delta^z J_\vp \Lvp(x_n)=f_n(-iz)$. Let us show that the sequence $(\delta^t J_\vp\Lvp(x_n))_{n\in\NN}$ converges to $\nu^{\frac{it}{2}} J_\vp\Lvp(x\delta^t)$:
\[\begin{split}
&\quad\;
\delta^t J_\vp \Lvp(x_n)=f_n(-it)=
\sqrt{\tfrac{n}{\pi}} \int_{\RR} e^{-n(p-it)^2} \nu^{-\frac{p-it}{2}}
\delta^{-ip}J_\vp \Lvp(x)\md p\\
&=
\nu^{\frac{it}{2}}
\sqrt{\tfrac{n}{\pi}} \int_{\RR} e^{-n(p-it)^2} \nu^{-\frac{p}{2}}
\delta^{-ip}J_\vp \Lvp(x)\md p=
\nu^{\frac{it}{2}} J_\vp \Lvp(x_n \delta^t)
\xrightarrow[n\to\infty]{} \nu^{\frac{it}{2}} J_\vp \Lvp(x\delta^t).
\end{split}\]
Norm closedness of $\delta^t$ implies $J_\vp \Lvp(x)\in \Dom(\delta^t)$ and $\delta^t J_\vp \Lvp(x)=\nu^{\frac{it}{2}} J_\vp \Lvp(x\delta^t)$.
\end{proof}

In what follows we introduce a space $\mc{D}_0$ of sufficiently nice vectors on which calculation \eqref{eq14} is justified and which forms a core for the operators involved. First, define
\[
\delta_{n,z}=\sqrt{\tfrac{n}{\pi}} \int_{\RR} e^{-nt^2}\nu^{zt} \delta^{it} \md t\in \Linf\quad(n\in\NN,z\in\CC).
\]
Note that for each $z\in\CC$, the sequence $(\delta_{n,z})_{n\in\NN}$ is bounded and converges to $\I$ in \sot. Next, for $x\in \mf{N}_\vp\cap{\mf{N}_\vp}^*\cap \mf{N}_\psi\cap {\mf{N}_\psi}^*, k\in\NN,A=(A_1,A_2)\in \CC^2$ define 
\[
x_{k,A}=\tfrac{k}{\pi} \int_{\RR}\int_{\RR}
e^{-k(t-A_1)^2-k(s-A_2)^2}\sigma^{\vp}_{t}\circ \sigma^{\psi}_{s}(x)\md t \md s \in\Linf.
\]
Finally, define a subspace $\mc{D}_0$ via
\[
\mc{D}_0=\lin\{\Lambda_\psi(\delta_{n,z}x_{k,A} \delta_{m,w})\,|\, 
x,x^*\in \mf{N}_\vp\cap \mf{N}_\psi, n,m,k\in\NN,A\in \CC^2,z,w\in \CC\}.
\]

\begin{lemma}\label{lemat16}$ $
\begin{itemize}
\item The subspace $\mc{D}_0$ is a core for $\nabla_\vp^{-\frac{1}{2}}$. Moreover, for $\xi\in \Dom(\nabla_\vp^{-\frac{1}{2}})$ we can find a sequence $(\xi_p)_{p\in\NN}$ in
\[
\{\Lambda_\psi(x_{k,A} \delta_{m,w})\,|\, 
x,x^*\in \mf{N}_\vp\cap \mf{N}_\psi, m,k\in\NN,A\in \CC^2,w\in\CC\}
\]
such that $\xi_p\xrightarrow[p\to\infty]{} \xi$ and $\nabla_\vp^{-\frac{1}{2}}\xi_p\xrightarrow[p\to\infty]{} \nabla_\vp^{-\frac{1}{2}}\xi$
\item Each element of $\mc{D}_0$ can be written as $\Lambda_\psi(x)$ for some $x\in\Linf$ such that $x,x^*\in \mf{N}_\vp\cap \mf{N}_\psi\cap \bigcap_{z\in\CC}\Dom(\sigma^\psi_z)$. Moreover, $\sigma^{\psi}_z(x)\in \mf{N}_\psi$ and $\Lambda_\psi(x^*),\Lambda_\psi(\sigma^{\psi}_z(x))\in\mc{D}_0$. Next, $\Lambda_\psi(x)\in \bigcap_{z\in\CC} \Dom(\nabla_\psi^z)$ and $
\nabla_\psi^{iz}\Lambda_\psi(x)=
\Lambda_{\psi}(\sigma_z^{\psi}(x))$.
\item For all $z,w\in\CC,\Lambda_\psi(x)\in\mc{D}_0$ the operator $\delta^z\circ x \circ\delta^w$ is closable and after closure belongs to $\mf{N}_\psi\cap \mf{N}_\vp$.
\item We have $J_\vp\mc{D}_0= \mc{D}_0$.
\end{itemize}
\end{lemma}

A proof of the above lemma requires only standard reasoning, hence will be skipped. In the next two lemmas we prove properties of $\mc{D}_0$ which allows us to derive the polar decomposition of $T'$.

\begin{lemma}\label{lemat3}
The subspace $\mc{D}_0$ is a core for $\nu^{-\frac{i}{4}}J_\vp   \delta^{-\frac{1}{2}} J_\vp$. We have
\[
\nu^{-\frac{i}{4}}J_\vp  \delta^{-\frac{1}{2}} J_\vp \Lambda_\psi(x)=\Lvp(x)
\]
for all $x\in\mf{N}_\vp\cap\mf{N}_\psi$ such that $x\circ\delta^{-\frac{1}{2}}$ is closable and $x\delta^{-\frac{1}{2}}\in \mf{N}_\psi$. Moreover, the operator
\[
(J_\vp\nabla_\vp^{\frac{1}{2}}) \circ (\nu^{-\frac{i}{4}}J_\vp   \delta^{-\frac{1}{2}} J_\vp )=
\nu^{\frac{i}{4}} (\nabla_\vp^{-\frac{1}{2}}\circ\delta^{-\frac{1}{2}})J_\vp
\]
is closable and $\mc{D}_0$ is a core for its closure $\nu^{\frac{i}{4}} \nabla_\vp^{-\frac{1}{2}}\delta^{-\frac{1}{2}}J_\vp$.
\end{lemma}

\begin{proof}
It is clear that $\lin \bigcup_{n\in\NN} \delta_{n,0}\LdG$ is a core for $\delta^{-\frac{1}{2}}$. Take $\xi=\delta_{n,0} \eta\in \Dom(\delta^{-\frac{1}{2}})$ for some $n\in\NN$ and let $(\eta_p)_{p\in\NN}$ be a sequence of vectors of the form $\Lambda_\psi(x_{k,A,B}\delta_{m,w})$ (see the first point of the Lemma \ref{lemat16}) converging to $\eta$. We have $\delta_{n,0}\eta_p\in\mc{D}_0$,
\[
\|\xi-\delta_{n,0}\eta_p\|\le\|\eta-\eta_p\|\xrightarrow[p\to\infty]{}0\quad\textnormal{and}\quad
\|\delta^{-\frac{1}{2}}\xi-\delta^{-\frac{1}{2}} \delta_{n,0}\eta_p\|\le \|\delta^{-\frac{1}{2}}\delta_{n,0} \| \|\eta-\eta_p\|\xrightarrow[p\to\infty]{}0,
\]
which shows that $\mc{D}_0$ is a core for $\delta^{-\frac{1}{2}}$. Since $\mc{D}_0$ is invariant under $J_\vp$, it is also a core for $\nu^{-\frac{i}{4}}J_\vp  \delta^{-\frac{1}{2}}J_\vp$.\\
Take $x\in \mf{N}_\vp\cap\mf{N}_\psi$ such that $x\circ\delta^{-\frac{1}{2}}$ is closable and $x\delta^{-\frac{1}{2}}\in \mf{N}_\psi$. Lemma \ref{lemat17} gives us $J_\vp\Lambda_\psi(x)\in\Dom(\delta^{-\frac{1}{2}})$ and $\nu^{-\frac{i}{4}} J_\vp \delta^{-\frac{1}{2}}J_\vp\Lambda_\psi(x)=
\Lambda_\psi(x\delta^{-\frac{1}{2}})=\Lvp(x).$\\
Equality from the claim $(J_\vp\nabla_\vp^{\frac{1}{2}}) \circ (\nu^{-\frac{i}{4}}J_\vp \delta^{-\frac{1}{2}} J_\vp )=
\nu^{\frac{i}{4}} (\nabla_\vp^{-\frac{1}{2}}\circ\delta^{-\frac{1}{2}})J_\vp $ is a straightforward consequence of the relation $J_\vp \nabla_\vp^{\frac{1}{2}}=\nabla_\vp^{-\frac{1}{2}}J_\vp$.\\
To deduce the last assertion let us observe that Lemma \ref{lemat4} gives us an equality $\nu^{i/8} \nabla_\vp^{-\frac{1}{2}}\delta^{-\frac{1}{2}}=\nu^{-i/8}\delta^{-\frac{1}{2}}\nabla_\vp^{-\frac{1}{2}}$. It follows that the closure of $\nu^{-i/4} \delta^{-\frac{1}{2}} \circ \nabla_\vp^{-\frac{1}{2}} $ is $\nabla_\vp^{-\frac{1}{2}}\delta^{-\frac{1}{2}}$.
Take $\xi\in $\\$\Dom(\nu^{-i/4} \delta^{-\frac{1}{2}}\circ\nabla^{-\frac{1}{2}}_\vp)$. For each $n\in\NN$ we have $\delta_{n,0}\xi\in \Dom( \nu^{-i/4}\delta^{-\frac{1}{2}}\circ \nabla^{-\frac{1}{2}}_\vp)$,
\begin{equation}\label{eq12}
\delta_{n,0}\xi\xrightarrow[n\to\infty]{}\xi
\end{equation}
 and
\begin{equation}\label{eq13}
\begin{split}
&\quad\;\nu^{-i/4}\delta^{-\frac{1}{2}}\circ \nabla^{-\frac{1}{2}}_\vp\;
(\delta_{n,0}\xi)=
\nu^{-i/4}\sigma^{\vp}_{i/2}(\delta_{n,0}) 
\delta^{-\frac{1}{2}}\circ \nabla^{-\frac{1}{2}}_\vp(\xi)\\
&=
\nu^{-i/4}\delta_{n,-1/2}
\delta^{-\frac{1}{2}}\circ \nabla^{-\frac{1}{2}}_\vp(\xi)
\xrightarrow[n\to\infty]{}
\nu^{-i/4}\delta^{-\frac{1}{2}}\circ \nabla^{-\frac{1}{2}}_\vp(\xi).
\end{split}\end{equation}
As previously, since $\mc{D}_0$ is invariant for $J_\vp$, it is enough to check that $\mc{D}_0$ is a core for $\nabla_\vp^{-\frac{1}{2}} \delta^{-\frac{1}{2}}$. Take $\xi\in \Dom(\nabla_\vp^{-\frac{1}{2}}  \delta^{-\frac{1}{2}})$. The above reasoning and equations \eqref{eq12}, \eqref{eq13} show that it is enough to take vector of the form $\xi=\delta_{n,0}\eta$ for $\eta\in\Dom(\nu^{-i/4}\delta^{-\frac{1}{2}}\circ \nabla_\vp^{-\frac{1}{2}} )$ and some $n\in\NN$. Let $(\eta_p)_{p\in \NN}$ be a sequence of vectors of the form $\Lambda_\psi(x_{k,A,B}\delta_{m,w})$ such that $
\eta_p\xrightarrow[p\to\infty]{} \eta$ and $
\nabla_\vp^{-\frac{1}{2}}\eta_p\xrightarrow[p\to\infty]{} \nabla_\vp^{-\frac{1}{2}}\eta.$
We have $\delta_{n,0}\eta_p\in\mc{D}_0$, $\delta_{n,0}\eta_p\xrightarrow[p\to\infty]{} \delta_{n,0}\eta=\xi$ and
\[\begin{split}
&\quad\;\|\nu^{-i/4} \delta^{-\frac{1}{2}}\circ\nabla_\vp^{-\frac{1}{2}} (\delta_{n,0}\eta-
\delta_{n,0}\eta_p)\|=
\|\delta_{n,-1/2} \delta^{-\frac{1}{2}} \circ \nabla_{\vp}^{-\frac{1}{2}} (\eta-\eta_p)\|\\
&\le
\|\delta_{n,-1/2} \delta^{-\frac{1}{2}} \|\| \nabla_{\vp}^{-\frac{1}{2}} (\eta-\eta_p)\|\xrightarrow[p\to\infty]{}0.
\end{split}\]
\end{proof}

\begin{lemma}\label{lemat15}
The subspace $\mc{D}_0$ is a core for $T'$.
\end{lemma}

\begin{proof}
Take $x\in \mf{N}_\psi\cap{\mf{N}_\vp}^*$ and define $x_n$ as
$
x_n=\tfrac{n}{\pi} \int_{\RR}\int_{\RR}e^{-n(r^2+p^2)}
\delta^{ip}x\delta^{ir}\md r \md p\quad(n\in\NN).
$
We have $x_n,x_n^*\in \mf{N}_\vp\cap \mf{N}_\psi$. Next, define
$
x_{n,n}=\tfrac{n}{\pi} \int_{\RR}\int_{\RR}e^{-n(t^2+s^2)}
\sigma^{\vp}_t\circ \sigma^{\psi}_s(x_n)\md s \md t.
$
 We have $\delta_{n,0}x_{n,n}\delta_{n,0}\in\mf{N}_\psi\cap {\mf{N}_\vp}^*,\Lambda_{\psi}(\delta_{n,0} x_{n,n}\delta_{n,0})\in\mc{D}_0$, $\Lambda_\psi(\delta_{n,0}x_{n,n}\delta_{n,0})\xrightarrow[n\to\infty]{} \Lambda_\psi(x)$ and
\[
T'\Lambda_\psi(\delta_{n,0}x_{n,n}\delta_{n,0})=\Lambda_{\vp}(\delta_{n,0}x_{n,n}^*\delta_{n,0})\xrightarrow[n\to\infty]{}
\Lvp(x^*)=T'\Lvp(x).
\]
\end{proof}

Now we can derive the main results of this section.

\begin{proposition}\label{stw11}
We have $(J_\vp \nabla_\vp^{\frac{1}{2}})\circ (\nu^{-\frac{i}{4}}J_\vp 
\delta^{-\frac{1}{2}} J_\vp)=
\nu^{\frac{i}{4}}(\nabla_\vp^{-\frac{1}{2}}\circ
\delta^{-\frac{1}{2}}) J_\vp$ and after closure
\[
\nu^{\frac{i}{4}} \, \nabla^{-\frac{1}{2}}_\vp \delta^{-\frac{1}{2}}\, J_\vp = T'.
\]
\end{proposition}

\begin{proof}
The first equality was justified in Lemma \ref{lemat3}. Take $\Lambda_\psi(x)\in \mc{D}_0$. Lemmas \ref{lemat16}, \ref{lemat3} justify the following calculation:
\[\begin{split}
&\quad\;
(J_\vp \nabla_\vp^{\frac{1}{2}})\circ (\nu^{-\frac{i}{4}}J_\vp 
\delta^{-\frac{1}{2}} J_\vp)\;\Lambda_\psi(x)=
J_\vp \nabla_\vp^{\frac{1}{2}} \Lambda_{\vp}(x)=
\Lvp(x^*)=
T'\Lambda_\psi(x).
\end{split}\]
In lemmas \ref{lemat3}, \ref{lemat15} we have shown that $\mc{D}_0$ is a core for $T'$ and $\nu^{\frac{i}{4}} \nabla_\vp^{-\frac{1}{2}}\delta^{-\frac{1}{2}}J_\vp$, which shows $T'= \nu^{\frac{i}{4}} \nabla_\vp^{-\frac{1}{2}} \delta^{-\frac{1}{2}} J_\vp$.
\end{proof}

The above result has a number of interesting corollaries.

\begin{corollary}
The polar decomposition of $T'$ is $T'=  (\nu^{\frac{i}{8}}J_\vp)\,(J_\vp\nu^{\frac{i}{8}} \nabla_\vp^{-\frac{1}{2}} \delta^{-\frac{1}{2}} J_\vp)$. Moreover, we have
\begin{equation}\label{eq15}
(J_\vp\nu^{\frac{i}{8}} \nabla_\vp^{-\frac{1}{2}} \delta^{-\frac{1}{2}} J_\vp)^{it}=
\nu^{ \frac{i}{8}t^2}J_\vp \nabla_\vp^{it/2}\delta^{it/2} J_\vp\quad(t\in\RR).
\end{equation}
\end{corollary}

\begin{proof}
The first equality follows directly from Proposition \ref{stw11}. Let us justify that it is indeed the polar decomposition. First, it is clear that $\nu^{\frac{i}{8}}J_\vp$ is antiunitary. Next, Lemma \ref{lemat4} implies that $\nu^{\frac{i}{8}} \nabla_\vp^{-\frac{1}{2}}\delta^{-\frac{1}{2}}$ is selfadjoint and strictly positive. Consequently, the operator $J_\vp\nu^{\frac{i}{8}} \nabla_\vp^{-\frac{1}{2}} \delta^{-\frac{1}{2}} J_\vp$ has the same properties. Uniqueness of the polar decomposition gives us the first claim. The second formula follows from Lemma \ref{lemat4}:
\[\begin{split}
&\quad\;
(J_\vp\nu^{\frac{i}{8}} \nabla_\vp^{-\frac{1}{2}} \delta^{-\frac{1}{2}} J_\vp)^{it}=
f(\nu^{\frac{i}{8}} \nabla_\vp^{-\frac{1}{2}} \delta^{-\frac{1}{2}})^{it}=
f((\nu^{\frac{i}{8}} \nabla_\vp^{-\frac{1}{2}} \delta^{-\frac{1}{2}})^{it})\\
&=
J_\vp (\nu^{\frac{i}{8}} \nabla_\vp^{-\frac{1}{2}} \delta^{-\frac{1}{2}})^{-it} J_\vp=
J_\vp  \nu^{-\frac{i}{8} t^2} \nabla_\vp^{it/2} \delta^{it/2}J_\vp,
\end{split}\]
where $f\colon a \mapsto J_\vp a^* J_\vp$.
\end{proof}

Now we combine our polar decomposition of $T'$ with the result of Caspers (Proposition \ref{stw5}) and Proposition \ref{stw1}. 

\begin{corollary}\label{wn3}
We have $\mc{Q}_L \nu^{\frac{i}{8}} J_\vp \mc{Q}_R^*=\Sigma$
and $\mc{Q}_R^*\mc{Q}_L=\mc{Q}_L^*\mc{Q}_R=\nu^{-\frac{i}{8}} J_{\hvp}J_\vp$.
\end{corollary}

Formula $\mc{Q}_R^*\mc{Q}_L=\nu^{-\frac{i}{8}} J_{\hvp} J_{\vp}$ is of great importance and will be used numerous times throught the paper.

\section{Operators expressed on the level of direct integrals}\label{secoperators}
In this section we will derive several equations, which express important operators on $\LdG$ via $\mc{Q}_L,\mc{Q}_R$ as direct integrals. The first result of this type comes from the polar decomposition of $T'$.
\begin{proposition}\label{stw6}
For all $t\in\RR$ we have
\[\begin{split}
\nabla_\psi^{it}\delta^{-it}=J_\vp   \nabla_\vp^{it}\delta^{it} J_\vp &=
\nu^{-\frac{i}{2}t^2}
\mc{Q}_R^*\bigl(\int_{\IrrG}^{\oplus} D_\pi^{2it}\otimes (E_\pi^{-2it})^{\msf T}\md\mu(\pi)\bigr)\mc{Q}_R,\\
J_\vp\nabla_\psi^{it} \delta^{-it} J_\vp=\nabla_{\vp}^{it} \delta^{it}&=
\nu^{\frac{i}{2}t^2} \mc{Q}_L^*\bigl(\int_{\IrrG}^{\oplus}
E_\pi^{2it}\otimes(D_\pi^{-2it})^{\msf T}\md\mu(\pi)\bigr)
\mc{Q}_L,\\
\nabla_\vp^{-it}\delta^{-it}=J_\vp \nabla_\psi^{-it}\delta^{it} J_\vp&=\nu^{\frac{i}{2}t^2}
\mc{Q}_R^*\bigl(
\int_{\IrrG}^{\oplus} E_\pi^{2it}\otimes (D_\pi^{-2it})^{\msf T}\md\mu(\pi)\bigr)\mc{Q}_R,\\
J_\vp \nabla_\vp^{-it}\delta^{-it} J_\vp=\nabla_\psi^{-it} \delta^{it}&=
\nu^{-\frac{i}{2}t^2} \mc{Q}_L^*\bigl(
\int_{\IrrG}^{\oplus} D_\pi^{2it}\otimes (E_\pi^{-2it})^{\msf T}\md\mu(\pi)\bigr)\mc{Q}_L.
\end{split}\]
\end{proposition}

\begin{proof}
First, observe that we have
$\nabla_{\psi}^{it}=J_{\hvp}\nabla_{\vp}^{-it}J_{\hvp} =\delta^{it}(J_\vp \delta^{it} J_\vp) \nabla_\vp^{it}$ (see \cite[Theorem 5.18]{Daele} and equation \eqref{eq20}). It follows that
\[
\nabla_\psi^{it} \delta^{-it}=\nu^{-it^2}\delta^{-it}\nabla_\psi^{it}=
\nu^{-it^2} J_\vp \delta^{it} \nabla_\vp^{it} J_\vp=
J_\vp  \nabla_\vp^{it} \delta^{it}J_\vp,
\]
and first equation in each row easily follows. The formula expressing $J_\vp \nabla_\vp^{it} \delta^{it} J_\vp$ via direct integral of operators follows from equation \eqref{eq15} combined with Proposition \ref{stw5}. The second equation can be found using already derived relation $\mc{Q}_L^*\mc{Q}_R=\nu^{-\frac{i}{8}} J_{\hvp}J_\vp$:
\[\begin{split}
J_\vp\nabla_\vp^{it} \delta^{it} J_\vp &=
\nu^{-\frac{i}{2}t^2}
\nu^{\frac{i}{8}} J_{\vp} J_{\hvp} \mc{Q}_L^*\bigl(
\int_{\IrrG}^{\oplus} D_\pi^{2it}\otimes (E_\pi^{-2it})^{\msf T}\md\mu(\pi)\bigr)\mc{Q}_L \nu^{-\frac{i}{8}} J_{\hvp} J_{\vp}\\
&=
\nu^{-\frac{i}{2} t^2}J_\vp
\mc{Q}_L^*\bigl(\int_{\IrrG}^{\oplus}
E_\pi^{2it}\otimes(D_\pi^{-2it})^{\msf T}\md\mu(\pi)\bigr)\mc{Q}_L J_\vp,
\end{split}\]
which implies $\nabla_{\vp}^{it} \delta^{it}=
\nu^{\frac{i}{2}t^2} \mc{Q}_L^*(\int_{\IrrG}^{\oplus}
E_\pi^{2it}\otimes(D_\pi^{-2it})^{\msf T}\md\mu(\pi))
\mc{Q}_L$. The last two equations comes from applying the operation $J_{\hvp} \cdot J_{\hvp}$ to both sides of already derived formulas.
\end{proof}

Let us now derive an interesting corollary of these results.

\begin{corollary}\label{wn4}
There exists a unique measurable function $f\colon \IrrG\rightarrow \RR_{>0}$ such that
\[\begin{split}
J_\vp\mc{Q}_R^*\bigl(\int_{\IrrG}^{\oplus} D_\pi^{2it}\otimes \I_{\ov{\msf{H}_\pi}}\md\mu(\pi)\bigr)^*\mc{Q}_RJ_\vp&=
\mc{Q}_R^*\bigl(\int_{\IrrG}^{\oplus} 
f(\pi)^{it}E_\pi^{2it}\otimes \I_{\ov{\msf{H}_\pi}}\md\mu(\pi)\bigr)\mc{Q}_R,\\
J_\vp\mc{Q}_R^*\bigl(\int_{\IrrG}^{\oplus} \I_{\msf{H}_\pi}\otimes (E_\pi^{2it})^{\msf T}\md\mu(\pi)\bigr)^*\mc{Q}_RJ_\vp&=
\mc{Q}_R^*\bigl(\int_{\IrrG}^{\oplus} 
f(\pi)^{it}\I_{\msf{H}_\pi}\otimes (D_\pi^{2it})^{\msf T}\md\mu(\pi)\bigr)\mc{Q}_R,\\
J_\vp\mc{Q}_L^*\bigl( \int_{\IrrG}^{\oplus}\I_{\msf{H}_\pi} \otimes (D_\pi^{2it})^{\msf T}\md\mu(\pi)\bigr)^*\mc{Q}_L J_\vp&=
\mc{Q}_L^*\bigl( 
\int_{\IrrG}^{\oplus} f(\pi)^{it} \I_{\msf{H}_\pi} \otimes (
E_\pi^{2it})^{\msf T}\md\mu(\pi)\bigr)\mc{Q}_L,\\
J_\vp\mc{Q}_L^*\bigl( \int_{\IrrG}^{\oplus} E_\pi^{2it}\otimes\I_{\ov{\msf{H}_\pi}}\md\mu(\pi)\bigr)^*\mc{Q}_L J_\vp&=
\mc{Q}_L^*\bigl( 
\int_{\IrrG}^{\oplus} f(\pi)^{it} 
D_\pi^{2it}\otimes\I_{\ov{\msf{H}_\pi}}\md\mu(\pi)\bigr)\mc{Q}_L
\end{split}\]
for all $t\in\RR$.
\end{corollary}

We note that the function $f$ might depend on the choice of a measure $\mu$.

\begin{proof}
Fix $t\in\RR$. The first and the third row in Proposition \ref{stw6} implies
\[
J_\vp\mc{Q}_R^*\bigl(\int_{\IrrG}^{\oplus} D_\pi^{2it}\otimes (E_\pi^{-2it})^{\msf T}\md\mu(\pi)\bigr)\mc{Q}_RJ_\vp=
\mc{Q}_R^*\bigl(\int_{\IrrG}^{\oplus}
E_\pi^{-2it}\otimes(D_\pi^{2it})^{\msf T}\md\mu(\pi)\bigr)\mc{Q}_R.
\]
Since $J_\vp \Linfd J_\vp=\Linfd$, $J_\vp {\Linfd}'J_\vp={\Linfd}'$ and the center of $\int_{\IrrG}^{\oplus} \I_{\msf{H}_\pi}\otimes \B(\ov{\msf{H}_\pi})\md\mu(\pi)$ is $\int_{\IrrG}^{\oplus} \CC \I_{\HS(\msf{H}_\pi)}\md\mu(\pi)$, Proposition \ref{stw7} implies that there exists a measurable function $f_t\colon\IrrG\rightarrow\TT$ such that
\[\begin{split}
&\quad\;J_\vp\mc{Q}_R^*\bigl(\int_{\IrrG}^{\oplus} D_\pi^{2it}\otimes \I_{\ov{\msf{H}_\pi}}\md\mu(\pi)\bigr)\mc{Q}_RJ_\vp
\mc{Q}_R^*\bigl(\int_{\IrrG}^{\oplus}
E_\pi^{2it}\otimes\I_{\ov{\msf{H}_\pi}}\md\mu(\pi)\bigr)\mc{Q}_R\\
&=
J_\vp\mc{Q}_R^*\bigl(\int_{\IrrG}^{\oplus} \I_{\msf{H}_\pi}\otimes (E_\pi^{2it})^{\msf T}\md\mu(\pi)\bigr)\mc{Q}_RJ_\vp
\mc{Q}_R^*\bigl(\int_{\IrrG}^{\oplus}
\I_{\msf{H}_\pi}\otimes(D_\pi^{2it})^{\msf T}\md\mu(\pi)\bigr)\mc{Q}_R\\
&=\int_{\IrrG}^{\oplus}\ov{ f_t(\pi)}\I_{\HS(\msf{H}_\pi)}\md\mu(\pi).
\end{split}\]
The above equations imply
\begin{equation}\label{eq16}
J_\vp\mc{Q}_R^*\bigl(\int_{\IrrG}^{\oplus} D_\pi^{2it}\otimes \I_{\ov{\msf{H}_\pi}}\md\mu(\pi)\bigr)^*\mc{Q}_RJ_\vp=
\mc{Q}_R^*\bigl(\int_{\IrrG}^{\oplus} 
f_t(\pi)E_\pi^{2it}\otimes \I_{\ov{\msf{H}_\pi}}\md\mu(\pi)\bigr)\mc{Q}_R
\end{equation}
and
\begin{equation}\label{eq17}
J_\vp\mc{Q}_R^*\bigl(\int_{\IrrG}^{\oplus} \I_{\msf{H}_\pi}\otimes (E_\pi^{2it})^{\msf T}\md\mu(\pi)\bigr)^*\mc{Q}_RJ_\vp=
\mc{Q}_R^*\bigl(\int_{\IrrG}^{\oplus} 
f_t(\pi)\I_{\msf{H}_\pi}\otimes (D_\pi^{2it})^{\msf T}\md\mu(\pi)\bigr)\mc{Q}_R.
\end{equation}
Equation \eqref{eq16} together with relation $\mc{Q}_L^*\mc{Q}_R=\nu^{-\frac{i}{8}} J_{\hvp}J_{\vp}$ (Corollary \ref{wn3}) gives us
\[\begin{split}
&\quad\;
J_\vp J_{\vp} J_{\hvp}\mc{Q}_L^*\bigl(\int_{\IrrG}^{\oplus} D_\pi^{2it}\otimes \I_{\ov{\msf{H}_\pi}}\md\mu(\pi)\bigr)^*\mc{Q}_L J_{\hvp} J_{\vp}J_\vp\\
&=
J_{\vp} J_{\hvp}\mc{Q}_L^*\bigl(\int_{\IrrG}^{\oplus} 
f_t(\pi)E_\pi^{2it}\otimes \I_{\ov{\msf{H}_\pi}}\md\mu(\pi)\bigr)\mc{Q}_L J_{\hvp} J_{\vp},
\end{split}\]
hence also (thanks to $\mc{Q}_LJ_{\hvp}\mc{Q}_L^*=\Sigma$, see Proposition \ref{stw1})
\[
\mc{Q}_L^*\bigl( \int_{\IrrG}^{\oplus}\I_{\msf{H}_\pi} \otimes (D_\pi^{-2it})^{\msf T}\md\mu(\pi)\bigr)^*\mc{Q}_L=
J_{\vp} \mc{Q}_L^*\bigl( 
\int_{\IrrG}^{\oplus} \ov{f_t(\pi)} \I_{\msf{H}_\pi} \otimes (
E_\pi^{-2it})^{\msf T}\md\mu(\pi)\bigr)\mc{Q}_L J_\vp.
\]
The last equation can be derived from equation \eqref{eq17} in a similar manner. Clearly we have $f_t(\pi)=f(\pi)^{it}$ for a measurable function $f\colon \IrrG\rightarrow \RR_{>0}$.
\end{proof}

In the second part of this section, we will transport operators $\nabla_{\hvp}^{it},\nabla_{\hpsi}^{it},\hat{\delta}^{it}\,(t\in\RR)$ to $\int_{\IrrG}^{\oplus} \HS(\msf{H}_\pi)\md\mu(\pi)$. We start with a formula expressing the action of $(\tau_t)_{t\in\RR}$ on matrix coefficients.
\begin{lemma}\label{lemat7}
For $\xi,\eta\in\int_{\IrrG}^{\oplus}\msf{H}_\pi\md\mu(\pi)$ and $t\in\RR$ we have
\[
\begin{split}
\tau_t(M^L_{\xi,\eta})&=
\nu^{-\frac{1}{2} it^2} \delta^{-it}
\int_{\IrrG}
(\id\otimes\omega_{D^{2it}_{\pi}\pi(\hat{\delta}_u^{-it})\xi_\pi,E_{\pi}^{2it}\eta_\pi})(U^{\pi *})\md\mu(\pi)\\
&=
\nu^{-\frac{1}{2} it^2}
\int_{\IrrG}
(\id\otimes\omega_{D_{\pi}^{-2it}\xi_\pi,E_{\pi}^{-2it}\pi(\hat{\delta}_u^{-it})\eta_\pi})(U^{\pi *})\md\mu(\pi)
\delta^{it},\\
\tau_t(M^R_{\xi,\eta})&=
\nu^{\frac{1}{2} it^2} 
\int_{\IrrG}
(\id\otimes\omega_{E^{2it}_{\pi}\xi_\pi,D_{\pi}^{2it}\pi(\hat{\delta}_u^{-it})\eta_\pi})(U^{\pi})\md\mu(\pi)
\delta^{it}\\
&=
\nu^{\frac{1}{2} it^2}
\delta^{-it}\int_{\IrrG}
(\id\otimes\omega_{E_{\pi}^{-2it}\pi(\hat{\delta}_u^{-it})\xi_\pi,D_{\pi}^{-2it}\eta_\pi})(U^{\pi})\md\mu(\pi).
\end{split}
\]
\end{lemma}

Later on in Proposition \ref{stw8} we will get simpler expressions for this action  (once we find out what $\pi(\hat{\delta}_u^{it})$ is).

\begin{proof}
The proof is based on several facts from the theory of locally compact quantum groups. First of all, we know that $\hat{\delta}^{it}=P^{-it}\nabla_{\psi}^{-it}$ (equation \eqref{eq20}). Next, \cite[Lemma 5.14]{Daele} gives us
\[
(\sigma^{\vp}_t\otimes\id) \mrW=(\I\otimes P^{-it})\mrW (\I\otimes \nabla_{\psi}^{-it}),\quad
(\sigma^{\psi}_t\otimes\id)\mrW=(\I\otimes \nabla_{\psi}^{-it})\mrW(\I\otimes P^{-it}),
\]
and $(\tau_t\otimes\id)\mrW=(\id\otimes\hat{\tau}_{-t})\mrW$. We note also that $\hat{\delta}^{it}\in \M(\CGD)$, $\hat{\delta}_u^{it}\in\M(\CGDu)$ and $\Lambda_{\whG}(\hat{\delta}_u^{it})=\hat{\delta}^{it}$ (\cite{Kustermans}). Fix $t\in\RR$, a representation $\pi\in\IrrG$ which factorises through $\CG$ (i.e.~$\pi=\pi'\circ \Lambda_{\whG}$ for a representation $\pi'\colon\CGD\rightarrow \B(\msf{H}_\pi)$) and arbitrary vectors $\xi_\pi,\eta_\pi\in\msf{H}_\pi$. We have
\[\begin{split}
&\quad\; 
\tau_t((\id\otimes\omega_{\xi_\pi,\eta_\pi})(U^{\pi *}))=
(\id\otimes\omega_{\xi_\pi,\eta_\pi})(\tau_t\otimes\id)(\id\otimes\pi)(\WW^*)\\
&=
(\id\otimes\omega_{\xi_\pi,\eta_\pi}\circ \pi')(\tau_t\otimes\id)(\mrW^*)=
(\id\otimes\omega_{\xi_\pi,\eta_\pi}\circ \pi')(\id\otimes\hat{\tau}_{-t})(\mrW^*)\\
&=
(\id\otimes\omega_{\xi_\pi,\eta_\pi}\circ \pi')
((\I\otimes P^{-it})(\mrW^*) (\I\otimes P^{it})).
\end{split}
\]
Now we write the above expression in two different ways: we have
\begin{equation}\label{eq1}
\begin{split}
&\quad\;
\tau_t((\id\otimes\omega_{\xi_\pi,\eta_\pi})(U^{\pi *}))=
(\id\otimes\omega_{\xi_\pi,\eta_\pi}\circ \pi')
((\I\otimes P^{-it}\nabla_{\psi}^{-it})
(\I\otimes \nabla_{\psi}^{it})(\mrW^*) 
(\I\otimes P^{it}))\\
&=
(\id\otimes\omega_{\xi_\pi,\eta_\pi}\circ \pi')
((\I\otimes\hat{\delta}^{it})\,(\sigma^{\vp}_t\otimes \id)(\mrW^*))=
\sigma^{\vp}_t((\id\otimes\omega_{\xi_\pi,\eta_\pi}\circ \pi)
((\I\otimes\hat{\delta}_{u}^{it})\,(\WW^*)))\\
&=
\sigma^{\vp}_t((\id\otimes\omega_{\pi(\hat{\delta}_u^{-it})\xi_\pi,\eta_\pi})
(U^{\pi *}))
 \end{split}\end{equation}
and
\begin{equation}\label{eq2}\begin{split}
&\quad\;
\tau_t((\id\otimes\omega_{\xi_\pi,\eta_\pi})(U^{\pi *}))=
(\id\otimes\omega_{\xi_\pi,\eta_\pi}\circ \pi')
((\I\otimes P^{-it})(\mrW^*) 
(\I\otimes \nabla_{\psi}^{-it})
(\I\otimes \nabla_{\psi}^{it}P^{it}))\\
&=
(\id\otimes\omega_{\xi_\pi,\eta_\pi}\circ\pi') (
(\sigma^{\psi}_{-t}\otimes\id)(\mrW^*) \, (\I\otimes \hat{\delta}^{-it}))
=
(\id\otimes\omega_{\xi_\pi,\eta_\pi}\circ\pi) (
(\sigma^{\psi}_{-t}\otimes\id)(\WW^*) \, (\I\otimes \hat{\delta}_{u}^{-it}))\\
&=
(\id\otimes\omega_{\xi_\pi, \pi(\hat{\delta}_{u}^{-it})\eta_\pi }) ((\sigma^{\psi}_{-t}\otimes\id)(U^{\pi *}))=
\sigma^{\psi}_{-t} (
(\id\otimes\omega_{\xi_\pi, \pi(\hat{\delta}_{u}^{-it})\eta_\pi })(U^{\pi *})).
 \end{split}\end{equation}
Let now $\xi,\eta$ be vectors in $\int_{\IrrG}^{\oplus}\msf{H}_\pi\md\mu(\pi)$. Then fields $(\pi(\hat{\delta}_u^{-it})\xi_\pi)_{\pi\in\IrrG}$, $(\pi(\hat{\delta}_u^{-it})\eta_\pi)_{\pi\in\IrrG}$ are also square integrable. Using equations \eqref{eq1}, \eqref{eq2} and Proposition \ref{stw4} we arrive at
\[\begin{split} 
&\quad\;
\tau_t ( M^{L}_{\xi,\eta})=
\tau_t( \int_{\IrrG} (\id\otimes\omega_{\xi_\pi,\eta_\pi})(U^{\pi *}) \md\mu(\pi))=
\int_{\IrrG}
\tau_t((\id\otimes\omega_{\xi_\pi,\eta_\pi})(U^{\pi *}))
\md\mu(\pi)\\
&=
\int_{\IrrG}
\sigma^{\vp}_t((\id\otimes\omega_{\pi(\hat{\delta}_u^{-it})\xi_\pi,\eta_\pi})(U^{\pi *}))\md\mu(\pi)=
\sigma^{\vp}_t(\int_{\IrrG}
(\id\otimes\omega_{\pi(\hat{\delta}_u^{-it})\xi_\pi,\eta_\pi})(U^{\pi *})\md\mu(\pi))\\
&=
\nu^{-\frac{1}{2} it^2} \delta^{-it}
\int_{\IrrG}
(\id\otimes\omega_{D^{2it}_{\pi}\pi(\hat{\delta}_u^{-it})\xi_\pi,E_{\pi}^{2it}\eta_\pi})(U^{\pi *})\md\mu(\pi)
\end{split}\]
and
\[\begin{split} 
&\quad\;
\tau_t ( M^{L}_{\xi,\eta})=
\tau_t( \int_{\IrrG} (\id\otimes\omega_{\xi_\pi,\eta_\pi})(U^{\pi *}) \md\mu(\pi))=
\int_{\IrrG}
\tau_t((\id\otimes\omega_{\xi_\pi,\eta_\pi})(U^{\pi *}))
\md\mu(\pi)
\\
&=
\int_{\IrrG}
\sigma^{\psi}_{-t}((\id\otimes\omega_{\xi_\pi,\pi(\hat{\delta}_u^{-it})\eta_\pi})(U^{\pi *}))\md\mu(\pi)=
\sigma^{\psi}_{-t}(\int_{\IrrG}
(\id\otimes\omega_{\xi_\pi,\pi(\hat{\delta}_u^{-it})\eta_\pi})(U^{\pi *})\md\mu(\pi))\\
&=
\nu^{-\frac{1}{2} it^2}
\int_{\IrrG}
(\id\otimes\omega_{D_{\pi}^{-2it}\xi_\pi,E_{\pi}^{-2it}\pi(\hat{\delta}_u^{-it})\eta_\pi})(U^{\pi *})\md\mu(\pi)
\delta^{it}.
\end{split}\]
The second pair of equations follow by applying the adjoint.
\end{proof}

Now we are ready to obtain the main results of this section. Even though we will prove them together, they are of different nature, hence we prefer to state them separately. First, we have a couple of equation expressing important operators on the level of direct integrals.

\begin{theorem}\label{tw1}
For every $t\in\RR$ we have
\[\begin{split}
\nabla_{\hpsi}^{-it}=\delta^{it}P^{it}&=
\mc{Q}_L^*
\bigl(\int_{\IrrG}^{\oplus} E_\pi^{2it} 
\otimes (E^{-2it}_\pi)^{\msf T}\md\mu(\pi)\bigr)\mc{Q}_L\\
&=
\mc{Q}_R^* 
\bigl( \int_{\IrrG}^{\oplus} 
D_{\pi}^{-2it} \otimes (D^{2 i t}_{\pi})^{\msf T} \md\mu(\pi)\bigr)
\mc{Q}_R,\\
\nabla_{\hvp}^{it}=J_\vp\delta^{it}P^{it}J_\vp &=
\mc{Q}_L^*
\bigl(\int_{\IrrG}^{\oplus} D_\pi^{-2it} 
 \otimes (D_{\pi}^{2it})^{\msf T}\md\mu(\pi)\bigr)\mc{Q}_L\\
&=
\mc{Q}_R^*
\bigl(\int_{\IrrG}^{\oplus} E_\pi^{2it} 
 \otimes (E_{\pi}^{-2it})^{\msf T}\md\mu(\pi)\bigr)\mc{Q}_R.
\end{split}
\]
\end{theorem}

Next, we show that the modular element for $\whG$ can be expressed using operators $(D_\pi)_{\pi\in\IrrG}, (E_\pi)_{\pi\in\IrrG}$.

\begin{proposition}\label{stw8}
For all $t\in\RR$ we have
\[\begin{split}
\hat{\delta}^{it}&=
\nu^{-\frac{i}{2} t^2} \mc{Q}_L^* \bigl(
\int_{\IrrG}^{\oplus} D_\pi^{2it} E_\pi^{-2it}\otimes\I_{\ov{\msf{H}_\pi}}\md\mu(\pi)\bigr)\mc{Q}_L\\
&=
\nu^{-\frac{i}{2}t^2}
\mc{Q}_R^*\bigl(\int_{\IrrG}^{\oplus}
\I_{\msf{H}_\pi}\otimes(D_\pi^{-2it} E_\pi^{2it})^{\msf T}\md\mu(\pi)\bigr)\mc{Q}_R.
\end{split}\]
Moreover, $\pi(\hat{\delta}_{u}^{it})=\nu^{\frac{i t^2}{2}}E_{\pi}^{-2it}D_{\pi}^{2it} $ and $\nu^{ist}D_{\pi}^{2is}E_{\pi}^{2it}=E_{\pi}^{2it}D_{\pi}^{2is}$ for all $s,t\in \RR$ and almost all $\pi\in \IrrG$. We also get better expressions for the action of $(\tau_t)_{t\in\RR}$:
\[
\begin{split}
\tau_t(M^L_{\xi,\eta})&=
\delta^{-it} M^L_{E^{2it}\xi,E^{2it}\eta}=
M^L_{D^{-2it}\xi,D^{-2it}\eta} \delta^{it},\\
\tau_t(M^R_{\xi,\eta})&=
M^{R}_{E^{2it}\xi,E^{2it}\eta}\,
\delta^{it}=
\delta^{-it}\,
M^R_{D^{-2it}\xi,D^{-2it}\eta}
\end{split}
\]
for all $t\in \RR$ and $\xi,\eta\in\int_{\IrrG}^{\oplus}\msf{H}_\pi\md\mu(\pi)$. 
\end{proposition}

\begin{proof}
Let $\xi,\eta$ be vector fields satisfying conditions from the first point of Lemma \ref{lemat5}. Note that vector fields $(D^{-2it}_\pi\xi_\pi)_{\pi\in\IrrG},(E_\pi^{-2it}\pi(\hat{\delta}_u^{-it})\eta_\pi)_{\pi\in\IrrG}$ also satisfy conditions of this lemma. Using the second equation from Lemma \ref{lemat7} we get:
\[\begin{split}
&\quad\;
\mc{Q}_L P^{it} \Lvp(M^L_{\xi,\eta})=
\nu^{\frac{t}{2}}\mc{Q}_L \Lvp(\tau_t(M^L_{\xi,\eta}))\\
&=
\nu^{\frac{t-it^2}{2}}\mc{Q}_L \Lvp( \int_{\IrrG}
(\id\otimes\omega_{D_{\pi}^{-2it}\xi_\pi,E_{\pi}^{-2it} \pi(\hat{\delta}_u^{-it})\eta_\pi})(U^{\pi *})\md\mu(\pi)
\,\delta^{it})\\
&=
\nu^{\frac{t-it^2}{2}}\mc{Q}_L 
J_\vp \sigma^{\vp}_{i/2}(\delta^{it})^* J_\vp\Lvp( \int_{\IrrG}
(\id\otimes\omega_{D_{\pi}^{-2it}\xi_\pi,E_{\pi}^{-2it}\pi(\hat{\delta}_u^{-it})\eta_\pi})(U^{\pi *})\md\mu(\pi)
)\\
&=
\nu^{\frac{-it^2}{2}}\mc{Q}_L 
 J_\vp\delta^{-it}J_\vp\mc{Q}_L^*
\int_{\IrrG}^{\oplus} E_{\pi}^{-2it} \pi(\hat{\delta}_u^{-it})\eta_\pi\otimes
\ov{ D_{\pi} D_{\pi}^{-2it} \xi_\pi } \md\mu(\pi)\\
&=
\nu^{\frac{-it^2}{2}}\mc{Q}_L 
 J_\vp\delta^{-it}J_\vp\mc{Q}_L^*
\bigl( \int_{\IrrG}^{\oplus}
E_{\pi}^{-2it} \pi(\hat{\delta}_u^{-it})\otimes (D_{\pi}^{2it} )^{\msf T}
\md\mu(\pi)\bigr)
\mc{Q}_L \Lvp(M^L_{\xi,\eta}).
 \end{split}
\]
Since the set of $\Lvp(M^L_{\xi,\eta})$ with $\xi,\eta$ as above form a lineary dense set (Lemma \ref{lemat13}), we get
\begin{equation}\label{eq8}
J_\vp\delta^{it}J_\vp P^{it}=
\mc{Q}_L^*\bigl(\int_{\IrrG}^{\oplus}
(\nu^{-\frac{it^2}{2}}
E_{\pi}^{-2it} \pi(\hat{\delta}_u^{-it}))\otimes (
D_{\pi}^{2it})^{\msf T} \md\mu(\pi)\bigr)\mc{Q}_L.
\end{equation}
Since $(J_\vp\delta^{it} J_\vp P^{it})_{t\in\RR}$, $( (D_\pi^{2it})^{\msf T})_{t\in\RR}$ are strongly continuous groups (see equation \eqref{eq20}) the same is true for $( \nu^{-\frac{it^2}{2}} E_\pi^{-2it} \pi(\hat{\delta}_u^{-it}))_{t\in\RR}$ (see point $2)$ of Lemma \ref{lemat2}). Using relations gathered in equation \eqref{eq20} one easily checks that $J_{\hvp}$ commutes with $J_{\vp} \delta^{it} J_{\vp} P^{it}$. Since $J_{\hvp}=\mc{Q}_L^* \Sigma\mc{Q}_L$, point $3)$ of Lemma \ref{lemat2} implies
\begin{equation}\label{eq3}
\nu^{-\frac{it^2}{2}} E_\pi^{-2it} \pi(\hat{\delta}_u^{-it})=
D_\pi^{-2it}\quad\Rightarrow\quad
\pi(\hat{\delta}_u^{it})=
\nu^{\frac{it^2}{2}} E_\pi^{-2it} D_\pi^{2it}
\quad(\pi\in\IrrG, t\in\RR)
\end{equation}
Let us choose $s,t\in \RR$ and use the fact that $( \pi(\hat{\delta}^{ip}_u))_{p\in\RR}$ is a group: we have
\[
\nu^{\frac{i(t+s)^2}{2}}E^{-2i(t+s)}_\pi D_\pi^{2i(t+s)}=
\pi(\hat{\delta}_u^{i(t+s)})=
\pi(\hat{\delta}_u^{it})
\pi(\hat{\delta}_u^{is})=
\nu^{\frac{it^2}{2}}E^{-2it}_\pi D_\pi^{2it}
\nu^{\frac{is^2}{2}}E^{-2is}_\pi D_\pi^{2is},
\]
and formula $\nu^{ist} E_\pi^{-2is} D_\pi^{2it}=
D_\pi^{2it} E_\pi^{-2is}$ easily follows. 
Equations expressing the action of $(\tau_t)_{t\in\RR}$ on matrix coefficients follows from the equation $\pi(\hat{\delta}_u^{it})=\nu^{\frac{i}{2}t^2} E_\pi^{-2it} D_\pi^{2it}$, commutation relation between $E_\pi^{it}$ and $D_\pi^{is}$ and Lemma \ref{lemat7}.
Let us now plug in the above results to equation \eqref{eq8}:
\begin{equation}\begin{split}\label{eq18}
J_\vp\delta^{it}J_\vp P^{it}&=
\nu^{-\frac{it^2}{2}}\mc{Q}_L^*\bigl(\int_{\IrrG}^{\oplus}
E_{\pi}^{-2it} \pi(\hat{\delta}_u^{-it})\otimes (
D_{\pi}^{2it})^{\msf T} \md\mu(\pi)\bigr)\mc{Q}_L\\
&=
\nu^{-\frac{it^2}{2}}\mc{Q}_L^*\bigl(\int_{\IrrG}^{\oplus}
E_{\pi}^{-2it} \nu^{\frac{it^2}{2}}E_\pi^{2it}D_\pi^{-2it}
\otimes (
D_{\pi}^{2it})^{\msf T} \md\mu(\pi)\bigr)\mc{Q}_L\\
&=
\mc{Q}_L^*\bigl(\int_{\IrrG}^{\oplus}
D_\pi^{-2it}
\otimes (
D_{\pi}^{2it})^{\msf T} \md\mu(\pi)\bigr)\mc{Q}_L,
 \end{split}\end{equation}
which is the third equation of Theorem \ref{tw1}. If we use formula $\mc{Q}_R^*\mc{Q}_L =\nu^{-\frac{i}{8}} J_{\hvp} J_{\vp}$, we readily get the second equation. Now we can derive the first pair of equations of Proposition \ref{stw8}. Since for all $t\in\RR$ we have $\nabla_\psi^{it}=\hat{\delta}^{-it} P^{-it}$ and $J_\vp\hat{\delta}^{it}=\hat{\delta}^{it} J_\vp$, it follows that $
\hat{\delta}^{it}=J_\vp \hat{\delta}^{it} J_\vp=
(J_\vp P^{-it} \delta^{-it} J_\vp) (J_\vp \delta^{it}\nabla_\psi^{-it}J_\vp),$
which we can express using equation \eqref{eq18} and Proposition \ref{stw6}:
\[\begin{split}
\mc{Q}_L\hat{\delta}^{it}\mc{Q}_L^*&=
\bigl( \int_{\IrrG}^{\oplus} D_\pi^{2it}\otimes(D_\pi^{-2it})^{\msf T}
\md\mu(\pi)\bigr) \nu^{-\frac{i}{2}t^2}
\bigl(\int_{\IrrG}^{\oplus}
E_\pi^{-2it} \otimes (D_\pi^{2it})^{\msf T}\md\mu(\pi)\bigr)\\
&=
\nu^{-\frac{i}{2} t^2} \int_{\IrrG}^{\oplus}
D_\pi^{2it} E_\pi^{-2it}\otimes \I_{\ov{\msf{H}_\pi}}\md\mu(\pi).
\end{split}\]
On the other hand, we also have $\hat{\delta}^{it}=( \nabla_\psi^{-it}\delta^{it} )(\delta^{-it}P^{-it})$, hence
\[
\mc{Q}_R \hat{\delta}^{it} \mc{Q}_R^*=
\nu^{-\frac{i}{2}t^2}
\bigl(\int_{\IrrG}^{\oplus} D_\pi^{-2it}\otimes(E_\pi^{2it})^{\msf T}\md\mu(\pi)\bigr)
\bigl( \int_{\IrrG}^{\oplus} D_\pi^{2it}\otimes (D_\pi^{-2it})^{\msf T}\md\mu(\pi)\bigr),
\]
which implies the second equation for $\hat{\delta}^{it}$ and ends the proof of Proposition \ref{stw8}. In order to finish the proof of Theorem \ref{tw1} we have to derive a lemma concerning the function $f$ introduced in Corollary \ref{wn4}.
\begin{lemma}\label{lemat19}
For all $t\in\RR$ we have
\[\begin{split}
J_{\vp} \mc{Q}_L^*\bigl( \int_{\IrrG}^{\oplus} f(\pi)^{it}\, \I_{\HS(\msf{H}_\pi)}\md\mu(\pi)\bigr)^* \mc{Q}_L J_\vp&=
\mc{Q}_L^*\bigl( \int_{\IrrG}^{\oplus} f(\pi)^{-it}\, \I_{\HS(\msf{H}_\pi)}\md\mu(\pi)\bigr) \mc{Q}_L,\\
J_{\vp} \mc{Q}_R^*\bigl( \int_{\IrrG}^{\oplus} f(\pi)^{it}\, \I_{\HS(\msf{H}_\pi)}\md\mu(\pi)\bigr)^* \mc{Q}_R J_\vp&=
\mc{Q}_R^*\bigl( \int_{\IrrG}^{\oplus} f(\pi)^{-it}\, \I_{\HS(\msf{H}_\pi)}\md\mu(\pi)\bigr) \mc{Q}_R.
\end{split}\]
\end{lemma}

\begin{proof}[Proof of Lemma \ref{lemat19}]
Recall that $J_\vp \hat{\delta}^{it} J_\vp=\hat{\delta}^{it}$, hence
\[
\nu^{\frac{i}{2} t^2} J_\vp\mc{Q}_L^* \bigl(
\int_{\IrrG}^{\oplus} D_\pi^{2it} E_\pi^{-2it}\otimes\I_{\ov{\msf{H}_\pi}}\md\mu(\pi)\bigr)\mc{Q}_LJ_\vp=
\nu^{-\frac{i}{2} t^2} \mc{Q}_L^* \bigl(
\int_{\IrrG}^{\oplus} D_\pi^{2it} E_\pi^{-2it}\otimes\I_{\ov{\msf{H}_\pi}}\md\mu(\pi)\bigr)\mc{Q}_L.
\]
Using the above relation and the fourth equation of Corollary \ref{wn4} we get
\[\begin{split}
&\quad\;
J_\vp \mc{Q}_L^*\bigl( \int_{\IrrG}^{\oplus} D_\pi^{2it} \otimes \I_{\ov{\msf{H}_\pi}} \md\mu(\pi) \bigr)^*\mc{Q}_L J_\vp\\
&=
J_\vp \mc{Q}_L^*\bigl( \int_{\IrrG}^{\oplus} D_\pi^{-2it}E_\pi^{2it} \otimes \I_{\ov{\msf{H}_\pi}} \md\mu(\pi) \bigr)\mc{Q}_L J_\vp\,
J_\vp \mc{Q}_L^*\bigl( \int_{\IrrG}^{\oplus} E_\pi^{2it} \otimes \I_{\ov{\msf{H}_\pi}} \md\mu(\pi) \bigr)^*\mc{Q}_L J_\vp\\
&=
\nu^{-it^2}\mc{Q}_L^*\bigl( \int_{\IrrG}^{\oplus} D_\pi^{-2it}E_\pi^{2it} \otimes \I_{\msf{H}_\pi} \md\mu(\pi) \bigr)\mc{Q}_L \,
\mc{Q}_L^*\bigl( \int_{\IrrG}^{\oplus} f(\pi)^{it}D_\pi^{2it} \otimes \I_{\ov{\msf{H}_\pi}} \md\mu(\pi) \bigr)\mc{Q}_L\\
&=
\mc{Q}_L^*\bigl( \int_{\IrrG}^{\oplus} f(\pi)^{it} 
E_\pi^{2it}
\otimes \I_{\ov{\msf{H}_\pi}} \md\mu(\pi) \bigr)\mc{Q}_L,
\end{split}\]
consequently
\[\begin{split}
&\quad\;
\mc{Q}_L^* \bigl(
\int_{\IrrG}^{\oplus} E_\pi^{2it}\otimes\I_{\ov{\msf{H}_\pi}}
\md\mu(\pi)\bigr)\mc{Q}_L=
J_\vp\bigl( J_\vp \mc{Q}_L^* \bigl(
\int_{\IrrG}^{\oplus} E_\pi^{2it}\otimes\I_{\ov{\msf{H}_\pi}}
\md\mu(\pi)\bigr)^*\mc{Q}_L J_\vp \bigr)^*J_\vp\\
&=
J_\vp \mc{Q}_L^* \bigl(
\int_{\IrrG}^{\oplus} f(\pi)^{it} D_\pi^{2it}\otimes\I_{\ov{\msf{H}_\pi}}
\md\mu(\pi)\bigr)^*\mc{Q}_L J_\vp\\
&=
\mc{Q}_L^* \bigl(
\int_{\IrrG}^{\oplus} f(\pi)^{it} E_\pi^{2it}\otimes\I_{\ov{\msf{H}_\pi}}
\md\mu(\pi)\bigr)\mc{Q}_L
J_\vp\mc{Q}_L^* \bigl(
\int_{\IrrG}^{\oplus} f(\pi)^{it} \I_{\HS(\msf{H}_\pi)}
\md\mu(\pi)\bigr)^*\mc{Q}_L J_\vp
\end{split}\]
and
\[
J_\vp\mc{Q}_L^* \bigl(
\int_{\IrrG}^{\oplus} f(\pi)^{it} \I_{\HS(\msf{H}_\pi)}
\md\mu(\pi)\bigr)^*\mc{Q}_L J_\vp
=
\mc{Q}_L^*\bigl(
\int_{\IrrG}^{\oplus} f(\pi)^{-it} \I_{\HS(\msf{H}_\pi)} \md\mu(\pi)\bigr)\mc{Q}_L.
\]
The second equation can be proved analogously or using equation $\mc{Q}_R^*\mc{Q}_L=\nu^{-\frac{i}{8}} J_{\hvp} J_{\vp}$.
\end{proof}
Using the above lemma and Corollary \ref{wn4} we can derive the first equation of Theorem \ref{tw1} out of the third one:
\[\begin{split}
&\quad\;\delta^{it} P^{it}= J_\vp J_\vp \delta^{it} P^{it} J_\vp J_\vp=
J_\vp \mc{Q}_L^* \bigl( \int_{\IrrG}^{\oplus}
D_\pi^{-2it} \otimes (D_\pi^{2it})^{\msf T}\md\mu(\pi)\bigr) \mc{Q}_LJ_\vp\\
&=
J_\vp \mc{Q}_L^* \bigl( \int_{\IrrG}^{\oplus}
f(\pi)^{it}\I_{\HS(\msf{H}_\pi)}\md\mu(\pi)\bigr) \mc{Q}_LJ_\vp
J_\vp \mc{Q}_L^* \bigl( \int_{\IrrG}^{\oplus}
f(\pi)^{-it}D_\pi^{-2it} \otimes (D_\pi^{2it})^{\msf T}\md\mu(\pi)\bigr) \mc{Q}_LJ_\vp\\
&=
\mc{Q}_L^* \bigl( \int_{\IrrG}^{\oplus}
f(\pi)^{it} \I_{\HS(\msf{H}_\pi)}\md\mu(\pi)\bigr) \mc{Q}_L
 \mc{Q}_L^* \bigl( \int_{\IrrG}^{\oplus}
f(\pi)^{-it}E_\pi^{2it} \otimes (E_\pi^{-2it})^{\msf T}\md\mu(\pi)\bigr) \mc{Q}_L\\
&=
\mc{Q}_L^* \bigl( \int_{\IrrG}^{\oplus}
E_\pi^{2it} \otimes (E_\pi^{-2it})^{\msf T}\md\mu(\pi)\bigr) \mc{Q}_L.
\end{split}\]
Now, the last equation of Theorem \ref{tw1} follows as usual from the formula relating $\mc{Q}_L$ and $\mc{Q}_R$. This concludes the proof of Theorem \ref{tw1} and Proposition \ref{stw8}.
\end{proof}

The commutation relation $\nu^{ist} D_\pi^{2is}E_\pi^{2it}=E_\pi^{2it}D_\pi^{2is}\,(t,s\in\RR)$ derived in the previous proposition has the following consequence.

\begin{corollary}\label{wn1}
If $\nu\neq 1$ then for almost all $\pi\in\IrrG$, operators $D_\pi,E_\pi$ have empty point spectrum. In particular, if $\nu\neq 1$ then the set of finite dimensional irreducible representations is of measure zero.
\end{corollary}

\section{Special cases}\label{secspecial}
In this section we show how the properties of operators $(E_\pi)_{\pi\in\IrrG},(D_\pi)_{\pi\in\IrrG}$ are related to the modular theory of a type I, second countable locally compact quantum group (i.e.~properties of the modular element, scaling group, modular automorphism groups etc.). First, let us mention three lemmas which are probably well known to experts and which hold for a general locally compact group.
 
\begin{lemma}\label{lemat11}
The following conditions are equivalent:
\begin{enumerate} [label=\arabic*)]
\item $P^{it}\in \Linf'$ for all $t\in \RR$,
\item the scaling group of $\GG$ is trivial,
\item $P^{it}=\I$ for all $t\in\RR$.
 \end{enumerate}
\end{lemma}

\begin{proof}
Implications $1)\Leftrightarrow 2)\Leftarrow 3)$ follow from the equation $\tau_t(x)=P^{it} x P^{-it}\,(x\in\Linf)$. For all $x\in\mf{N}_{\vp}$ and $t\in\RR$ we have $P^{it}\Lvp(x)=\nu^{\frac{t}{2}}\Lvp(\tau_t(x))$, hence $2)$ implies $P^{it}=\nu^{\frac{t}{2}}\I$. Taking the norm of both sides gives us $1=\nu^{\frac{t}{2}}$ hence $\nu=1$.
\end{proof}

\begin{lemma}\label{lemat12}$ $
\begin{enumerate}[label=\arabic*)]
\item The Haar integrals on $\GG$ are tracial if, and only if $P=\hat{\delta}=\I$.
\item $\whG$ is unimodular if, and only if $\nabla_{\vp}^{it}=\nabla_{\psi}^{-it}\,(t\in\RR)$.
\end{enumerate}
\end{lemma}

\begin{proof}
We will use formulas gathered in equation \eqref{eq20}. Equality $\nabla_{\psi}^{it}=\hat{\delta}^{-it}P^{-it}\,(t\in\RR)$ shows that $P=\hat{\delta}=\I$ implies $\nabla_{\psi}^{it}=\I$ and the traciality of $\psi$. Let us prove the converse impliation. If $\nabla_\psi^{it}=\I$ then $P^{it}=\hat{\delta}^{-it}\in\LL^{\infty}(\whG)$ for all $t\in\RR$. Since $P^{it}$ commutes with $J_{\hvp}$, we have $P^{it}=J_{\hvp}P^{it}J_{\hvp}\in\LL^{\infty}(\whG)'$ and by the previous lemma $P^{it}=\I=\hat{\delta}^{-it}$.\\
If $\whG$ is unimodular, then we have $J_{\hvp} \nabla_\vp^{-it} J_{\hvp}=\nabla_\psi^{it}=P^{-it}$ for all $t\in\RR$. Since $P^{-it}$ commutes with $J_{\hvp}$, it follows that $\nabla_\psi^{it}=\nabla_\vp^{-it}$. On the other hand, if $\nabla_\psi^{it}=\nabla_\vp^{-it}$ for all $t\in\RR$, then
\[
\hat{\delta}^{-it} P^{-it}=
\nabla_\psi^{it}=\nabla_\vp^{-it}=J_{\hvp} \nabla_{\psi}^{it} J_{\hvp}=
J_{\hvp}\hat{\delta}^{-it}P^{-it}J_{\hvp}=
J_{\hvp}\hat{\delta}^{-it}J_{\hvp} P^{-it}
\]
and we get $\hat{\delta}^{it}=J_{\hvp} \hat{\delta}^{it} J_{\hvp}$. This in particular means that $\hat{\delta}^{it}\in \mc{Z}(\Linfd)$ and \cite[Proposition 1.23]{TakesakiII} implies $\hat{\delta}^{it}=J_{\hvp} \hat{\delta}^{-it} J_{\hvp}$, unimodularity of $\whG$ follows.
\end{proof}

Although we will not use this result, let us mention here that if $\GG$ is unimodular then $\sigma^{\hvp}_t(x)=\hat{\tau}_t(x)=\hat{\delta}^{-\frac{it}{2}} x \hat{\delta}^{\frac{it}{2}}$ and $\Delta_{\whG}(\sigma^{\hvp}_t(x))=(\sigma^{\hvp}_t\otimes\sigma^{\hvp}_t)\Delta_{\whG}(x)$ for all $t\in\RR,x\in\Linfd$. It is a consequence of the formula $P^{-2it}=\delta^{it} (J_{\vp} \delta^{it} J_{\vp}) \hat{\delta}^{it} (J_{\hvp} \hat{\delta}^{it} J_{\hvp})$ and $\Delta_{\whG}(\hat{\delta}^{it})=\hat{\delta}^{it}\otimes\hat{\delta}^{it}$ (see \cite[Theorem 5.20, Proposition 5.15]{Daele}.

\begin{lemma}\label{lemat18}
For all $t,s\in\RR$, if $\sigma^\vp_t=\sigma^\psi_s$ then $\nabla_\vp^{it}=\nabla_\psi^{is}$. If $(s,t)\neq (0,0)$ then also $\nu=1$.
\end{lemma}

\begin{proof}
For all $x\in \mf{N}_\vp$ we have $\nabla_\psi^{-is} \nabla_\vp^{it}\Lvp(x)=\nu^{\frac{1}{2}s}\Lvp(\sigma^{\psi}_{-s}(\sigma^{\vp}_t(x)))=\nu^{\frac{1}{2}s} \Lvp(x)$ (see \cite[Remark 5.2 ii)]{Daele}), hence $\nabla_\psi^{-is}\nabla_\vp^{it}=\nu^{\frac{1}{2}s} \I$. Taking the norm of both sides implies $\nu^{\frac{1}{2} s}=1$ and proves the first claim. If $s\neq 0$ then we get $\nu=1$, if $s=0$ and $(s,t)\neq (0,0)$ then $t\neq 0$ and we get $\nabla_\vp^{it}=\I$. Formula $\nabla_\vp^{it}\Lambda_\psi(y)= \nu^{\frac{t}{2}} \Lambda_\psi(\sigma^{\vp}_t(y))=\nu^{\frac{t}{2}} \Lambda_\psi(y)\,(y\in\mf{N}_\psi)$ implies $\nu=1$.
\end{proof}
The next theorem is the main result of this section. It presents a web of connections between various properties of a type I, second countable locally compact quantum group (and its dual).

\begin{theorem}\label{tw2}
Let $\GG$ be a second countable, type I locally compact quantum group. Consider the following conditions:
\begin{enumerate}[label=\arabic *)]
\item $D_\pi^{it}\in \CC\I_{\msf{H}_\pi}$ for all $t\in\RR$ and almost all $\pi\in \IrrG$,
\item $E_\pi^{it}\in \CC\I_{\msf{H}_\pi}$ for all $t\in\RR$ and almost all $\pi\in \IrrG$,
\item the Haar integrals on $\whG$ are tracial ( left $\Leftrightarrow$ right $\Leftrightarrow $ both),
\item the Haar integrals on $\GG$ are tracial ( left $\Leftrightarrow$ right $\Leftrightarrow $ both),
\item $\hat{\delta}^{it}\in \mc{Z}(\LL^{\infty}(\whG))$ for all $t\in\RR$,
\item $\GG$ is unimodular,
\item $E_\pi^{it}D_\pi^{-it}\in\CC \I_{\msf{H}_\pi}$ for all $t\in\RR$ and almost all $\pi$,
\item $E_\pi^{it} =D_\pi^{it}$ for all $t\in \RR$ and almost all $\pi\in \IrrG$,
\item $\whG$ is unimodular,
\item $E_\pi^{it} D_\pi^{it}\in \CC\I_{\msf{H}_\pi}$ for all $t\in\RR$ and almost all $\pi\in\IrrG$,
\item $\delta^{it}\in\mc{Z}(\Linf)$ for all $t\in\RR$,
\item $\sigma^{\vp}_t=\sigma^{\psi}_t$ for all $t\in\RR$.
\end{enumerate}
The following implications hold:
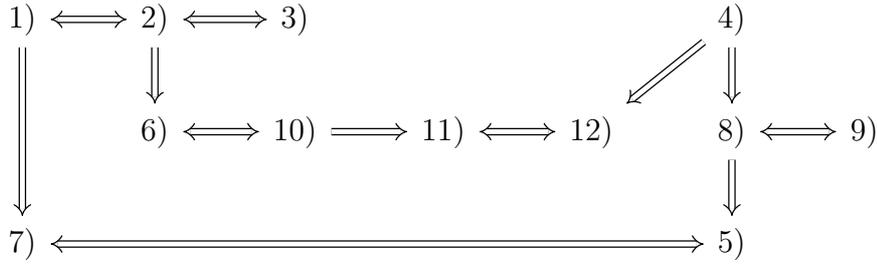
\[
\begin{tikzcd}
 1) \arrow[r,Leftrightarrow,""] 
\arrow[dd,Rightarrow,""]& 
2)
\arrow[d,Rightarrow,""]
 \arrow[r,Leftrightarrow,""]
 & 
3) 
&  && 4)
\arrow[dl,Rightarrow,""]
\arrow[d,Rightarrow,""] & 
\\
&
6) \arrow[r,Leftrightarrow,""]&
10) \arrow[r,Rightarrow,""]
& 11) \arrow[r,Leftrightarrow,""]&
12) &
8) \arrow[r,Leftrightarrow,""]
\arrow[d,Rightarrow,""] & 9)
&\\
7) \arrow[rrrrr,Leftrightarrow,""] &&&& &5)
& 
\end{tikzcd}
\]
Moreover, each of the above conditions implies $\nu=1$.
\end{theorem}

\begin{proof}
First, let us note that $\vp$ is tracial if and only $\psi$ is tracial: it is a consequence of the equation $\nabla_\psi^{it}=J_{\hvp} \nabla_{\vp}^{-it} J_{\hvp}\,(t\in\RR)$. Equivalence $1)\Leftrightarrow 2) \Leftrightarrow 3)$ is a part of the Desmedt's theorem, one can also deduce this from formulas for $\nabla_{\hvp},\nabla_{\hpsi}$ -- see Theorem \ref{tw1}. 
Equivalence $7)\Leftrightarrow 5)$ follows from the formula for $\hat{\delta}^{it}$ in Proposition \ref{stw8} and $\mc{Q}_L \Linfd \mc{Q}_L^*=\int_{\IrrG}^{\oplus} \B(\msf{H}_\pi)\otimes\I_{\ov{\msf{H}_\pi}}\md\mu(\pi)$ (see Proposition \ref{stw7}). Equivalence $8)\Leftrightarrow 9)$ is a straightforward consequence of Proposition \ref{stw8}.\\
Assume $6)$, i.e.~that $\GG$ is unimodular and let us derive $10)$. Fix $t\in\RR$. Theorem \ref{tw1} gives us
\[
P^{it}=
\mc{Q}_L^*\bigl(\int_{\IrrG}^{\oplus}
E_\pi^{2it}\otimes (E^{-2it}_\pi)^{\msf T}\md\mu(\pi)\bigr)\mc{Q}_L=
\mc{Q}_L^*\bigl(\int_{\IrrG}^{\oplus}
D_\pi^{-2it}\otimes (D^{2it}_\pi)^{\msf T}\md\mu(\pi)\bigr)\mc{Q}_L,
\]
which implies $E_\pi^{2it}\otimes (E^{-2it}_\pi)^{\msf T}=D_\pi^{-2it}\otimes
(D_\pi^{2it})^{\msf T}\,(\pi\in \IrrG)$. Consequently, $
D_\pi^{2it}E_\pi^{2it}S=SD_\pi^{2it}E_\pi^{2it}$ for all $S\in\HS(\msf{H}_\pi)$. This means that $D_\pi^{2it}E_\pi^{2it}=\lambda_t\I_{\msf{H}_\pi}$ for some $\lambda_t\in\CC$ and we arrive at the point $10)$. On the other hand, point $10)$ implies 
that there exists $\lambda_{t,\pi}\in \TT$ such that $E_\pi^{it} =\lambda_{t,\pi} D_\pi^{-it}$. It follows that $\nu=1$, moreover the first and the third row of Theorem \ref{tw1} implies $\delta^{it}=J_\vp \delta^{it} J_\vp$. This in particular means that $\delta^{it}$ belongs to the center of $\Linf$ -- we have $\delta^{it}=J_\vp (\delta^{it})^* J_\vp$ \cite[Proposition 1.23]{TakesakiII}. These two equations together imply $\delta=\I$.\\
The last equivalence, $11)\Leftrightarrow 12)$, follows easily from the formula $\sigma^{\psi}_t(x)=\delta^{it} \sigma^{\vp}_t(x)\delta^{-it}$ ($x\in\Linf,t\in\RR$, see \cite[Theorem 3.11]{Daele}).\\
The remaining implications are trivial. Let us now argue why all of the above conditions imply $\nu=1$. Clearly we only need to justify this for $7)$ and $11)$. If $E_\pi^{it} D_\pi^{-it}\in \CC\I_{\msf{H}_\pi}$ then $\nu^{ist}D_{\pi}^{2is}E_{\pi}^{2it}=E_{\pi}^{2it}D_{\pi}^{2is}$ forces $\nu=1$. If $\delta^{it}\in \mc{Z}(\Linf)$ then $\nu^{it^2}\delta^{it}=\sigma^\vp_t(\delta^{it})=\delta^{it}$ for all $t\in\RR$ (\cite[Proposition 1.23]{TakesakiII}), hence also in this case $\nu=1$.
\end{proof}

Let us now show how certain classes of quantum groups fit into the above diagram. In particular, these examples show that one-sided implications in the above theorem cannot be reversed.

\begin{proposition}
Let $\GG$ be a type I, second countable locally compact quantum group.
\begin{itemize}
\item If $\GG$ is classical and non-unimodular, then it satisfies $4)$ and does not satisfy $6)$.
\item If $\whG$ is classical and non-unimodular, then $\GG$ satisfies $3)$ and does not satisfy $9)$.
\item If $\GG$ is compact and not of Kac type, then it satisfies $6)$ and does not satisfy $5)$.
\item If $\GG$ is discrete and non-unimodular, then it satisfies $9)$ and does not satisfy $11)$.
\end{itemize}
\end{proposition}

The numbering in the above proposition corresponds to the numbering introduced in Theorem \ref{tw2}. Clearly each of the above classes is non-empty: examples are given by the classical $ax+b$ group, its dual, the $\SUd$ group and its dual (see Example \ref{su2d}). At the end of this section let us derive a corollary of Theorem \ref{tw2}.

\begin{corollary}
Let $\GG$ be a type I, second countable locally compact quantum group. The Haar integrals on $\GG$ and $\whG$ are tracial if, and only if $\GG$ and $\whG$ are unimodular.
\end{corollary}

\begin{proof}
The right implication is an easy corollary of Lemma \ref{lemat12}. Assume that $\GG$ and $\whG$ are unimodular. Equivalences $8)\Leftrightarrow 9)$ and $6)\Leftrightarrow 10)$ of Theorem \ref{tw2} imply that $E_\pi=D_\pi\in\CC\I_{\msf{H}_\pi}$ for almost all $\pi\in\IrrG$. Then $2)\Leftrightarrow 3)$ of the same theorem implies that the Haar integrals on $\whG$ are tracial. Equalities $\nabla_{\hpsi}^{it}=\delta^{-it} P^{-it},\;
\nabla_{\psi}^{it}=\hat{\delta}^{-it} P^{-it}\,(t\in\RR)$ end the proof.
\end{proof}

\section{Examples}\label{secexamples}

\subsection{Group $\widehat{\mathrm{SU}_q(2)}$}\label{su2d}
Fix a real number $q\in\left]-1,1\right[\setminus\{0\}$. Let $\GG=\mathrm{SU}_q(2)$ be the compact quantum group introduced by Woronowicz in \cite{Woronowiczsu2} and let $\bbGamma$ be the dual discrete quantum group $\bbGamma=\widehat{\mathrm{SU}_q(2)}$. The \cst-algebra of continuous functions on the quantum space $\SUd$, $\mathrm{C}(\SUd)$ is the universal unital \cst-algebra generated by elements $\alpha,\gamma$ satisfying the following relations:
\[\begin{split}
\alpha^*\alpha + \gamma^*\gamma=\I,\quad\alpha\gamma&=q\gamma\alpha,\quad
\alpha\gamma^*=q\gamma^*\alpha,\\
\alpha\alpha^*+q^2\gamma\gamma^*=\I,\quad
\gamma\gamma^*&=\gamma^*\gamma.
\end{split}\]
The Haar integral of $\SUd$ is faithful on $\mathrm{C}(\SUd)$ and we have $\mathrm{C}^{u}(\SUd)=\mathrm{C}(\SUd)$ ($\SUd$ is coamenable, see \cite[Theorem 2.12]{bmt}). Furthermore,  the \cst-algebra $\mathrm{C}(\SUd)$ is separable and type I (see \cite[Theorem A2.3]{Woronowiczsu2}) hence $\bbGamma$ is an interesting example of a second countable, type I discrete quantum group\footnote{In this section $\bbGamma$ is the "main" group and $\GG$ is the "dual" one.}.  We will describe the Plancherel measure for this group and show how various operators related to $\bbGamma$ act on the level of direct integrals. Let us start with describing the measurable space $\IrrGamma$ (i.e.~the spectrum of $\mathrm{C}(\SUd)$). The following result is a reformulation of \cite[Theorem 3.2]{Vaksman}:
\begin{proposition}
Measurable space $\IrrGamma$ can be identified with the disjoint union of two circles $\TT\sqcup \TT=\{\psi^{1,\rho}\,|\,\rho\in\TT\}\cup \{\psi^{2,\lambda}\,|\,\lambda\in\TT\}$. Representations $\psi^{1,\rho}$ are one dimensional and given by
\[
\psi^{1,\rho}(\alpha)=\rho,\quad
\psi^{1,\rho}(\alpha^*)=\ov{\rho},\quad
\psi^{2,\rho}(\gamma)=0,\quad
\psi^{2,\rho}(\gamma^*)=0\quad(\rho\in\TT).
\]
Representations $\psi^{2,\lambda}$ act on a separable Hilbert space $\msf{H}_\lambda=\ell^2(\ZZ_+)$ with an orthonormal basis $\{\phi_k\,|\, k\in\ZZ_+\}$ via
\begin{alignat*}{5}
&\psi^{2,\lambda}(\alpha)\phi_k&&=\sqrt{1-q^{2k}} \phi_{k-1},\quad&&
\psi^{2,\lambda}(\alpha^*)\phi_k&&=\sqrt{1-q^{2(k+1)}} \phi_{k+1}, &&\\
&\psi^{2,\lambda}(\gamma)\phi_k&&=\lambda q^k \phi_{k},\quad&&
\psi^{2,\lambda}(\gamma^*)\phi_k&&=\ov{\lambda} q^k \phi_{k},&&\quad\quad (\rho\in\TT,k\in\ZZ_+),
\end{alignat*}
with the convention $\phi_{-n}=0\,(n\in\NN)$.
\end{proposition}

In the next proposition we calculate the Plancherel measure of $\bbGamma$, the unitary operator $\mc{Q}_L$ and operators $(D_\pi)_{\pi\in\IrrGamma}$. In what follows, $\vp,\psi$ are the Haar integrals on $\bbGamma$ and $h$ is the Haar integral on $\GG=\SUd$.

\begin{proposition}
The Plancherel measure of $\bbGamma$ equals $0$ on $\{\psi^{1,\rho}\,|\,\rho\in\TT\}$ and the normalized Lebesgue measure on the second circle $\{\psi^{2,\lambda}\,|\,\lambda\in\TT\}$. Consequently, we will identify $\IrrGamma$ with $\TT$. Operators $\{D_\lambda\,|\, \lambda\in\TT\}$ are given by
\[
D_\lambda=(1-q^2)^{-\frac{1}{2}}\diag(1,|q|^{-1},|q|^{-2},\dotsc) \quad(\lambda\in\TT)
\]
with respect to the basis $\{\phi_k\,|\, k\in\ZZ_+\}$. Operator $\mc{Q}_L$ is given by
\[
\mc{Q}_L \colon\LdG\ni \Lambda_h(a)\mapsto\int_{\IrrGamma}^{\oplus}
\psi^{2,\lambda}(a) D_\lambda^{-1} \md\mu(\lambda)\in
\int_{\IrrGamma}^{\oplus} \HS(\msf{H}_\lambda) \md\mu(\lambda)\quad
(a\in \mathrm{C}(\SUd)).
\]
\end{proposition}

\begin{proof}
Define $\mu$ to be the normalized Lebesgue measure on the second circle of $\IrrGamma=\TT\sqcup \TT$ and let $\mc{Q}_L$ be the operator given by the above formula. In order to show that these objects are the one given by Desmedt's theorem, we will use\footnote{This result is formulated only for type I quantum groups with finite dimensional irreducible representations. However, its proof is based on \cite[Lemma 3.2]{Krajczok} and proof of this lemma works just as well for more general groups with bounded operators $D_\pi^{-1}$, such as second countable, type I discrete quantum groups.}  point $7)$ of \cite[Theorem 3.3]{Krajczok}.  Let us start with showing that $\mc{Q}_L$ is well defined and unitary. First, it is clear that for $a\in \mathrm{C}(\SUd)$ the field of operators $(\psi^{2,\lambda}(a)D_\lambda^{-1})_{\lambda\in \TT}$ is measurable and square integrable. Consequently, we can introduce a densely defined linear map $\mc{Q}_L\colon \Lh(a)\mapsto \int_{\IrrGamma}^{\oplus}\psi^{2,\lambda}(a) D_\lambda^{-1}\md\mu(\lambda)$. Since $\|\mc{Q}_L \Lh(a)\|\le \|a\|\,(a\in \mathrm{C}(\SUd))$, the linear map $\mc{Q}_L \circ \Lh$ is bounded. Let us now show that $\mc{Q}_L$ is isometry, i.e.~$
\ismaa{\mc{Q}_L \Lh(a')}{\mc{Q}_L \Lh(a)}= \ismaa{\Lh(a')}{\Lh(a)}
$ for all $a,a'\in \mathrm{C}(\SUd)$. Since
\[\begin{split}
&\quad\;
\ismaa{\mc{Q}_L \Lh(a')}{\mc{Q}_L \Lh(a)}=
\bigl\langle \int_{\IrrGamma}^{\oplus} \psi^{2,\lambda}(a') D_\lambda^{-1}\md\mu(\lambda) \big|\int_{\IrrGamma}^{\oplus} \psi^{2,\lambda}(a) D_\lambda^{-1}\md\mu(\lambda)\bigr\rangle\\
&=
\bigl\langle \int_{\IrrGamma}^{\oplus} \psi^{2,\lambda}(\I) D_\lambda^{-1}\md\mu(\lambda) \big|\int_{\IrrGamma}^{\oplus} \psi^{2,\lambda}(a'^*a) D_\lambda^{-1}\md\mu(\lambda)\bigr\rangle=
\ismaa{\mc{Q}_L\Lh(\I)}{\mc{Q}_L \Lh(a'^*a)}
\end{split}\]
and $\ismaa{\Lh(a')}{\Lh(a)}=\ismaa{\Lh(\I)}{\Lh(a'^*a)}$, it is enough to consider the case $a'=\I$. Next, as maps $\mc{Q}_L\circ \Lh,\Lh$ are bounded and linear, it is enough to consider $a$ in a basis of $\Pol(\SUd)$, $\{\alpha^l\gamma^n\gamma^{*m},\alpha^{*l'}\gamma^n\gamma^{*m}\,|\, l,n,m\in \ZZ_+,l'\in\NN\}$ (see \cite[Theorem 1.2]{Woronowiczsu2}).\\
In order to calculate $\ismaa{\Lh(\I)}{\Lh(a)}$ we need to introduce a faithful representation $\pi_0$\\$\colon \mathrm{C}(\SUd)\rightarrow \B(\ell^2(\ZZ_+\times\ZZ))$ defined in \cite{Woronowiczsu2}. One can express the Haar integral $h$ as
\[
h(a)=(1-q^2)\sum_{k=0}^{\infty}q^{2k} \ismaa{\phi_{k,0}}{\pi_0(a)\phi_{k,0}}
\quad(a\in \mathrm{C}(\SUd)),
\]
where $\{\phi_{k,p}\,|\,(k,p)\in\ZZ_+\times\ZZ\}$ is the standard basis of $\ell^2(\ZZ_+\times\ZZ)$. Now, for $l,n,m\in\ZZ_+$ we have
\[
\ismaa{ \Lh(\I)}{\Lh(\alpha^l \gamma^n \gamma^{*m})}=
h(\alpha^l \gamma^n \gamma^{*m})=\delta_{l,0}(1-q^2)
\sum_{k=0}^{\infty}q^{2k}\delta_{n,m} q^{(n+m)k}=
\delta_{l,0}\delta_{n,m}\tfrac{1-q^2}{1-q^{2(1+n)}}
\]
and similarly $\ismaa{\Lh(\I)}{\Lh(\alpha^{*l} \gamma^n \gamma^{*m})}=\delta_{l,0}\delta_{n,m}\tfrac{1-q^2}{1-q^{2(1+n)}}$. On the other hand
\[\begin{split}
&\quad\;\ismaa{\mc{Q}_L \Lh(\I)}{\mc{Q}_L \Lh(\alpha^l\gamma^n \gamma^{*m})}=\bigl\langle \int_{\IrrGamma}^{\oplus} D_\lambda^{-1}\md\mu(\lambda) \big|
\int_{\IrrGamma}^{\oplus} \psi^{2,\lambda}(\alpha^l \gamma^n \gamma^{*m})D_\lambda^{-1}\md\mu(\lambda)\bigr\rangle\\
&=
\delta_{l,0}(1-q^2)\int_{\IrrGamma} \sum_{k=0}^{\infty}\ismaa{\phi_k}{
\lambda^{n-m}q^{(n+m)k}q^{2k} \phi_{k}} \md\mu( \lambda)\\
&=
\delta_{l,0}\delta_{n,m}(1-q^2)\sum_{k=0}^{\infty} q^{(n+m)k}q^{2k}=
\delta_{l,0}\delta_{n,m} \tfrac{1-q^2}{1-q^{2(1+n)}}.
\end{split}\]
In an analogous manner we check $
\ismaa{\mc{Q}_L \Lh(\I)}{\mc{Q}_L \Lh(\alpha^{*l}\gamma^n \gamma^{*m})}=
\delta_{l,0}\delta_{n,m} \tfrac{1-q^2}{1-q^{2(1+n)}}$. This shows that $\mc{Q}_L$ is isometry and consequently extends to the whole of $\LdG$. Let us now argue that $\mc{Q}_L$ is surjective. Fix $\lambda\in \TT$, $k,l\in\ZZ_+$. We have $\psi^{2,\lambda}(\gamma\gamma^*)\phi_k=q^{2k} \phi_k$, hence $
\psi^{2,\lambda}(\chi_{\{q^{2l}\}}(\gamma\gamma^*)) \phi_k=
\delta_{k,l}\phi_k$
(note that operator $\chi_{\{q^{2l}\}}(\gamma\gamma^*)$ belongs to $\mathrm{C}(\SUd)$ because $q^{2l}$ is an isolated point in the spectrum of $\gamma\gamma^*$). Next, for $n \in\ZZ_+$ the following holds
\[
\psi^{2,\lambda}( \alpha^n \chi_{q^{2l}}(\gamma\gamma^*)) \phi_k=
\delta_{k,l}(\prod_{a=0}^{n-1} (1-q^{2(k-a)})^{\frac{1}{2}})\phi_{k-n}=
\delta_{k,l} (\prod_{a=0}^{n-1} (1-q^{2(k-a)})^{\frac{1}{2}})\phi_{l-n}
\]
which (together with a similar reasoning for $\alpha^*$) implies that for all $l,n\in\ZZ_+$ there exists an operator $E_{n,l}\in\mathrm{C}(\SUd)$ such that $\psi^{2,\lambda}(E_{n,l})\phi_k=\delta_{l,k}\phi_n\;(k\in\ZZ_+,\lambda\in\TT)$. Next, for $m\in\ZZ_+$ we have
\[
\psi^{2,\lambda}(q^{-lm} E_{n,l}\gamma^m)\phi_k=\delta_{l,k}\lambda^m \phi_n,\quad
\psi^{2,\lambda}(q^{-lm} E_{n,l}\gamma^{*m})\phi_k=\delta_{l,k}\lambda^{-m} \phi_n\quad(k\in\ZZ_+,\lambda\in\TT)
\]
and consequently for any polynomial function $P$ in $\lambda,\ov{\lambda}$ and $n,l\in\ZZ_+$ an operator\\$
\int_{\IrrGamma}^{\oplus} P(\lambda) \psi^{2,\lambda}(E_{n,l})
\md\mu(\lambda)$ belongs to the image of $\mc{Q}_L$. By density of such polynomials in $\LL^2(\TT)$ it follows that for all $f\in\LL^{2}(\TT)$ 
\begin{equation}\label{eq21}
\int_{\IrrGamma}^{\oplus} f(\lambda) \psi^{2,\lambda}(E_{n,l}) \md\mu(\lambda)\in \mc{Q}_L (\LdG).
\end{equation}
We have an isomorphism (given by choice of bases) $
\int_{\IrrGamma}^{\oplus} \HS(\msf{H}_\lambda)\md\mu(\lambda)\simeq
\LL^2(\TT)\otimes \HS(\ell^2(\ZZ_+)),$
hence it is clear that operators as in \eqref{eq21} span a dense subspace in $\int_{\IrrGamma}^{\oplus} \HS(\msf{H}_\lambda) \md\mu(\lambda)$, and consequently $\mc{Q}_L$ is unitary. Let us now check the first commutation relation. We have
\[\begin{split}
&\quad\;
\mc{Q}_L \lambda^{\bbGamma}(\omega) \mc{Q}_L^* (\mc{Q}_L \Lh(a))=
\mc{Q}_L \Lh(\lambda^{\bbGamma}(\omega) a)=
\int_{\IrrGamma}^{\oplus} \psi^{2,\lambda}(\lambda^{\bbGamma}(\omega)a)D_\lambda^{-1} \md\mu(\lambda)\\
&=\int_{\IrrGamma}^{\oplus} \psi^{2,\lambda}(\lambda^{\bbGamma}(\omega))\psi^{2,\lambda}(a) D_\lambda^{-1}\md\mu(\lambda)=
\bigl(\int_{\IrrGamma}^{\oplus} \psi^{2,\lambda}(\lambda^{\bbGamma}(\omega))\otimes\I_{\ov{\msf{H}_\lambda}}\md\mu(\lambda)\bigr)\mc{Q}_L \Lh(a),
\end{split}\]
for all $\omega\in\ell^1(\bbGamma),a\in\mathrm{C}(\SUd)$ where $\lambda^{\bbGamma}(\omega)=(\omega\otimes\id)\mrW^{\bbGamma}$, hence
\begin{equation}\label{eq9}
\mc{Q}_L \lambda^{\bbGamma}(\omega)\mc{Q}_L^*=
\int_{\IrrGamma}^{\oplus} \psi^{2,\lambda}(\lambda^{\bbGamma}(\omega))\otimes\I_{\ov{\msf{H}_\lambda}}\md\mu(\lambda)\quad(\omega\in\ell^1(\bbGamma)).
\end{equation}
In order to show the second commutation relation, let us show that $\mc{Q}_L$ transports $J_{h}$ to the direct integral of adjoints. For $a\in \Pol(\SUd)$ we have
\[
\mc{Q}_L J_{h} \Lh(a)=\mc{Q}_L \Lh(\sigma^h_{-i/2}(a^*))=
\int_{\IrrGamma}^{\oplus}\psi^{2,\lambda}(\sigma^h_{-i/2}(a^*))
D_\lambda^{-1}\md\mu(\lambda).
\]
Next, observe that  $\psi^{2,\lambda}(\sigma^h_t(a))=D_\lambda^{-2it} \psi^{2,\lambda}(a) D_\lambda^{2it}$ for all $\lambda\in\TT,t\in\RR,a\in\Pol(\SUd)$. Indeed, we have $\sigma^h_t(\alpha)=|q|^{-2it} \alpha,\sigma^h_t(\gamma)=\gamma\,(t\in\RR)$ (\cite[Example 1.7.4]{NeshTu}) and consequently
\[
\psi^{2,\lambda}(\sigma^h_t(\gamma))=\psi^{2,\lambda}(\gamma)=D_\lambda^{-2it} \psi^{2,\lambda}(\gamma) D_\lambda^{2it}\quad(t\in\RR)
\]
and similarly for all $k\in\ZZ_+,t\in\RR$
\[\begin{split}
&\quad\;
D_\lambda^{-2it} \psi^{2,\lambda}(\alpha) D_\lambda^{2it}\phi_k=
(1-q^{2k})^{\frac{1}{2}} |q|^{-2ikt} |q|^{2i(k-1)t} \phi_{k-1}
=
|q|^{-2it} \psi^{2,\lambda}(\alpha) \phi_k
=\psi^{2,\lambda}(\sigma^h_t(\alpha))\phi_k.
\end{split}\]
It follows that for all $a\in \Pol(\SUd)$
\[
\mc{Q}_L J_{h}\Lh(a)=
\int_{\IrrGamma}^{\oplus}
D_\lambda^{-1}\psi^{2,\lambda}(a^*) D_\lambda D_\lambda^{-1} \md\mu(\lambda)=
\int_{\IrrGamma}^{\oplus}
(\psi^{2,\lambda}(a) D_\lambda^{-1})^*\md\mu(\lambda),
\]
hence $\mc{Q}_L J_{h} \mc{Q}_L^*$ equals $\Sigma=\int^{\oplus}_{\IrrGamma} J_{\msf{H}_\lambda}\md\mu(\lambda)$. Now we can show the second commutation relation. Formula $
\chi(\mrV^{\bbGamma})=(J_{h}\otimes J_{h})(\mrW^{\bbGamma})^*(J_{h}\otimes J_{h})$ (\cite[Proposition 5.9]{Daele}) implies that for all $\omega\in \ell^1(\bbGamma)$ we have $
(\omega\otimes\id)\chi(\mrV^{\bbGamma})=J_{h} ((\omega\circ R^{\bbGamma}\otimes\id)\mrW^{\bbGamma})^* J_{h}$ and consequently
\[\begin{split}
&\quad\;
\mc{Q}_L (\omega\otimes\id)\chi(\mrV^{\bbGamma}) \mc{Q}_L^*=
\mc{Q}_L J_{h} \mc{Q}_L^* \bigl(
\int_{\IrrGamma}^{\oplus} \psi^{2,\lambda}(\lambda^{\bbGamma}(\omega\circ R^{\bbGamma}))\otimes\I_{\ov{\msf{H}_\lambda}}\md\mu(\lambda)\bigr)^*
\mc{Q}_L J_{h} \mc{Q}_L^*\\
&=
\int_{\IrrGamma}^{\oplus} \I_{\msf{H}_\lambda}\otimes
\psi^{2,\lambda}(\lambda^{\bbGamma}(\omega\circ R^{\bbGamma}))^{\msf T}
\md\mu(\lambda)=
\int_{\IrrGamma}^{\oplus} \I_{\msf{H}_\lambda}\otimes
(\psi^{2,\lambda})^c(\lambda^{\bbGamma}(\omega))
\md\mu(\lambda),
\end{split}\]
which is the second commutation relation. We are left to show
\[
\mc{Q}_L (\Linf\cap\Linf') \mc{Q}_L^*=
\Diag(\int_{\IrrGamma}^{\oplus}\HS(\msf{H}_\lambda)\md\mu(\lambda)),
\]
let us first argue that
\begin{equation}\label{eq10}
\mc{Q}_L \Linf \mc{Q}_L^*=
\int_{\IrrGamma}^{\oplus}\B(\msf{H}_\lambda)\otimes\I_{\ov{\msf{H}_\lambda}}\md\mu(\lambda)).
\end{equation}
Inclusion $\subseteq$ follows from the commutation relation \eqref{eq9}. On the other hand, equation \eqref{eq9} and reasoning similar to the one showing that $\mc{Q}_L$ is unitary, implies that for any polynomial $P$ in $\lambda,\ov{\lambda}$ and $n,l\in\ZZ_+$ we have
\[
\int_{\IrrGamma}^{\oplus} P(\lambda) \psi^{2,\lambda}(E_{n,l})\otimes\I_{\ov{\msf{H}_\lambda}}\md\mu(\lambda)\in \mc{Q}_L \Linf \mc{Q}_L^*.
\]
$\swot$ density of polynomials in $\LL^{\infty}(\TT)$ and isomorphism $
\int_{\IrrGamma}^{\oplus} \B(\msf{H}_\lambda)\otimes\I_{\ov{\msf{H}_\lambda}}\md\mu(\lambda)\simeq \LL^{\infty}(\TT)\bar\otimes \B(\ell^2(\ZZ_+))
$
gives us \eqref{eq10}. Consequently
\[\begin{split}
\mc{Q}_L (\Linf\cap\Linf')\mc{Q}_L^*&=
\bigl( \int_{\IrrGamma}^{\oplus} \B(\msf{H}_\lambda)\otimes\I_{\ov{\msf{H}_\lambda}}\md\mu(\lambda)\bigr)\cap
\bigl( \int_{\IrrGamma}^{\oplus} \I_{\msf{H}_\lambda}\otimes
\B(\ov{\msf{H}_\lambda})\md\mu(\lambda)\bigr)\\
&=
\Diag(\int_{\IrrGamma}^{\oplus}\HS(\msf{H}_\lambda)\md\mu(\lambda)).
\end{split}\]
\end{proof}

In the next proposition we find an action of the operator $P^{it}$ on the level of direct integrals.

\begin{proposition}\label{stw9}
For each $t\in\RR$, operator $\mc{Q}_L P^{it} \mc{Q}_L^*$ acts on $\int_{\IrrGamma}^{\oplus} \HS(\msf{H}_\lambda) \md\mu(\lambda)$ as follows:
\[
\mc{Q}_LP^{it}\mc{Q}_L^*\colon
 \int_{\IrrGamma}^{\oplus} T_\lambda \md\mu(\lambda)\mapsto 
\int_{\IrrGamma}^{\oplus}T_{\lambda\, |q|^{2it}}\md\mu(\lambda).
\]
\end{proposition}

Note that the above result implies that $\mc{Q}_L P^{it}\mc{Q}_L^*$ is not decomposable.

\begin{proof}
Let $\tilde{P}^{it}$ be the operator in the claim, i.e.~$\tilde{P}^{it}\colon \int_{\IrrGamma}^{\oplus} T_\lambda\md\mu(\lambda)\mapsto \int_{\IrrGamma}^{\oplus} T_{\lambda |q|^{2it} }\md\mu(\lambda)$. Clearly it is well defined and bounded. The scaling group of $\GG=\SUd$ acts as follows (\cite[Example 1.7.8]{NeshTu})
\[\begin{split}
&\tau^{\GG}_t(\alpha)=\alpha,\quad \tau^{\GG}_t(\alpha^*)=\alpha^*,
\quad\tau^{\GG}_t(\gamma)=|q|^{2it}\gamma,\quad
\tau^{\GG}_t(\gamma^*)=|q|^{-2it}\gamma^*\quad(t\in\RR).
\end{split}\]
Recall that $P^{it}$ satisfies $
P^{it} \Lh(a)=\Lh(\tau^{\GG}_t(a))$ for all $t\in\RR,a\in\mathrm{C}(\GG)$.
Fix $l,k,n,m\in\ZZ_+,\lambda\in\TT$ and corresponding operator $\alpha^{l}\gamma^n\gamma^{*m}$ in the basis of $\Pol(\GG)$. We have
\[\begin{split}
&\quad\;\psi^{2,\lambda}(\alpha^l \gamma^n \gamma^{*m})\phi_k=
 (\prod_{a=0}^{l-1} (1-q^{2(k-a)})^{\frac{1}{2}})
\lambda^{n-m} q^{k(n+m)}\phi_{k-l}\\
&=
|q|^{-2it(n-m)}
(\prod_{a=0}^{l-1} (1-q^{2(k-a)})^{\frac{1}{2}})
(\lambda\,|q|^{2it})^{n-m} q^{k(n+m)}\phi_{k-l}\\
&=
|q|^{-2it(n-m)} \psi^{2,\lambda |q|^{2it}} (
\alpha^l\gamma^n\gamma^{*m})\phi_k,
\end{split}\]
(recall that we use convention $\phi_{-p}=0$ for $p\in \NN$) and consequently
\[\begin{split}
&\quad\;
\mc{Q}_L P^{it}\Lh(\alpha^{l}\gamma^n \gamma^{*m})=
|q|^{2it(n-m)}\mc{Q}_L \Lh(\alpha^l \gamma^n \gamma^{*m})\\
&=
\int_{\IrrGamma}^{\oplus} \psi^{2,\lambda |q|^{2it}}(\alpha^{l}\gamma^n\gamma^{*m})
D_\lambda^{-1}\md\mu(\lambda)=
\tilde{P}^{it} \mc{Q}_L \Lh(\alpha^l\gamma^n\gamma^{*m}).
\end{split}\]
In a similar manner we check $\mc{Q}_L P^{it} \Lh(\alpha^{*l} \gamma^n \gamma^{*m})=
\tilde{P}^{it}\mc{Q}_L \Lh(\alpha^{*l} \gamma^n \gamma^{*m})$. The claim follows because $ \Lh(\Pol(\GG))$ is dense in $\LdG$.
\end{proof}

The last result of this section describes the action of an operator $\mc{Q}_L J_{\vp}\mc{Q}_L^*$.

\begin{proposition}\label{stw10}
Operator $\mc{Q}_L J_{\vp}\mc{}Q_L^*$ acts on $\int_{\IrrGamma}^{\oplus} \HS(\msf{H}_\lambda) \md\mu(\lambda)$ as follows:
\[
\mc{Q}_L J_{\vp}\mc{Q}_L^*\colon
 \int_{\IrrGamma}^{\oplus} T_\lambda \md\mu(\lambda)\mapsto 
\int_{\IrrGamma}^{\oplus}j_\lambda T_{-\sgn(q)\lambda} j_\lambda\md\mu(\lambda),
\]
where $j_\lambda$ is the antilinear operator on $\msf{H}_\lambda=\ell(\ZZ_+)$ given by $j_\lambda \phi_k=\phi_k\,(\lambda\in\TT,k\in\ZZ_+)$.
\end{proposition}

Note that this result implies that operator $\mc{Q}_L J_\vp \mc{Q}_L^*$ is not decomposable if $q>0$.

\begin{proof}
Using formula $R^{\GG}=S^{\GG}\tau_{i/2}^{\GG}$ and \cite[Equation 1.14]{Woronowiczsu2} one easily checks that
\[\begin{split}
&R^{\GG}(\alpha)=\alpha^*,\quad
R^{\GG}(\alpha^*)=\alpha,\quad
R^{\GG}(\gamma)=-\operatorname{sgn}(q)\gamma,\quad
R^{\GG}(\gamma^*)=-\operatorname{sgn}(q)\gamma^*.
\end{split}\]
On the other hand we have $R^{\GG}(a)=J_{\vp}a^* J_{\vp}$ for all $a\in \mathrm{C}(\SUd)$, hence
\[
J_{\vp}\alpha=\alpha J_{\vp},\quad
J_{\vp}\alpha^*=\alpha^* J_{\vp},\quad
J_{\vp}\gamma=-\operatorname{sgn}(q)\gamma^* J_{\vp},\quad
J_{\vp} \gamma^*=-\sgn(q) \gamma J_{\vp}.
\]
Denote by $\tilde{J}_{\vp}$ the operator from the claim and fix $\lambda\in\TT,k,n,m,l\in\ZZ_+$. We have
\[\begin{split}
&\quad\;
\psi^{2,\lambda}(\alpha^l \gamma^m \gamma^{* n}) \phi_k=
\lambda^{m-n} q^{(m+n)k} (\prod_{a=0}^{l-1}(1-q^{2(k-a)})^{\frac{1}{2}}) \phi_{k-l}\\
&=
(-\sgn(q))^{m+n}
(-\sgn(q)\lambda)^{m-n} q^{(m+n)k} (\prod_{a=0}^{l-1}(1-q^{2(k-a)})^{\frac{1}{2}}) \phi_{k-l}\\
&=
(-\sgn(q))^{m+n}
j_\lambda \psi^{2,-\sgn(q)\lambda} (\alpha^l \gamma^n \gamma^{* m})j_\lambda\phi_k,
\end{split}\]
consequently
\[\begin{split}
&\quad\;
\mc{Q}_L J_{\vp}  \Lh(\alpha^l\gamma^n\gamma^{*m})=
\mc{Q}_L\alpha^l (-\sgn(q))^n \gamma^{* n} (-\sgn(q))^m \gamma^m
J_{\vp}\Lh(\I)\\
&=
(-\sgn(q))^{n+m}
\int_{\IrrGamma}^{\oplus} \psi^{2,\lambda}(\alpha^l \gamma^{m} \gamma^{*n} )D_\lambda^{-1} \md\mu(\lambda)
\\
&=
\int_{\IrrGamma}^{\oplus}
j_\lambda \psi^{2,-\sgn(q)\lambda}(\alpha^l \gamma^n \gamma^{*m})
j_\lambda
D_\lambda^{-1}\md\mu(\lambda)
\\&=
\tilde{J}_{\vp} \int_{\IrrGamma}^{\oplus} \psi^{2,\lambda}(\alpha^l \gamma^n \gamma^{*m}) D_\lambda^{-1}\md\mu(\lambda)=
\tilde{J}_{\vp} \mc{Q}_L \Lh(\alpha^l \gamma^n \gamma^{*m}).
\end{split}\]
Equation $\mc{Q}_L J_\vp \Lh(\alpha^{* l}\gamma^n \gamma^{* m})=\tilde{J}_{\vp} \mc{Q}_L \Lh(\alpha^{* l}\gamma^n \gamma^{* m})$ can be checked similarly.
\end{proof}

\begin{remark}\label{uwaga1}
In propositions \ref{stw9}, \ref{stw10} we have expressed operators $P^{it}\,(t\in\RR)$ and $J_{\vp}$ on $\int_{\IrrGamma}^{\oplus} \HS(\msf{H}_\lambda)\md\mu(\lambda)$. Theorem \ref{tw1} and Proposition \ref{stw6} allow us to do the same for $\delta^{it},\nabla_\vp^{it},\nabla_\psi^{it}\,(t\in\RR)$ -- operators obtained in this way are not decomposable.
\end{remark}

\subsection{Quantum group $az+b$}
In this section we will describe some aspects of the theory of the quantum $az+b$ group. We begin by introducing a complex number $q$ and an abelian group $\Gamma_q\subseteq \CC^{\times}$. We will consider three cases:
\begin{enumerate}[label=\arabic*)]
\item $q=e^{\frac{2\pi i}{N}}$ for a natural number $N\in 2\NN\setminus\{2\}$ and $\Gamma_q=\{q^k r\,|\, k\in\ZZ,r\in\RR_{>0}\}$,
\item $q$ is a real number in $\left]0,1\right[$ and $\Gamma_q=\{q^{i\theta+k}\,|\, \theta\in\RR,k\in\ZZ\}$,
\item $q=e^{\frac{1}{\rho}}$, where $\operatorname{Re}(\rho)<0,\operatorname{Im}(\rho)=\frac{N}{2\pi}$ and $N\in 2\ZZ\setminus\{0\}$. In this case \\$\Gamma_q=\{e^{\frac{k+it}{\rho} }\,|\, k\in\ZZ,t\in\RR\}$.
\end{enumerate}

It will be more convenient for us to work in the dual picure\footnote{In fact, $\GG$ is isomorphic to the quantum group opposite to quantum $az+b$.}: let $\whG$ be the quantum $az+b$ group associated with the parameter $q$. We refer the reader to papers \cite{Woronowiczazb, Soltanazb, WoronowiczHaar} for construction of these groups, here we will recall only necessary properties.\\
We treat all three cases simultaneously. The group $\Gamma_q$ has closure given by $\ov{\Gamma}_q=\Gamma_q\cup\{0\}$ and is selfdual. This duality is implemented by a certain bicharacter $\chi\colon \Gamma_q\times\Gamma_q\rightarrow \TT$. We choose a Haar measure on $\Gamma_q$ in such a way that the Fourier transform $\mc{F}(f)(\gamma)=\int_{\Gamma_q} \chi(\gamma,\gamma')f(\gamma')\md\mu(\gamma')$ is a unitary operator on $\LL^2(\Gamma_q)$. Next, the group $\Gamma_q$ acts on $\mathrm{C}_0(\ov{\Gamma}_q)$ by translations: $\sigma_\gamma(f)(\gamma')=f(\gamma\gamma')\,(f\in \mathrm{C}_0(\ov{\Gamma}_q),\gamma\in\Gamma_q,\gamma'\in\ov{\Gamma}_q)$. Let $\mathrm{C}_0(\ov{\Gamma}_q)\rtimes_\sigma \Gamma_q\subseteq\B(\LL^2(\Gamma_q))$ be the associated crossed product \cst-algebra (note that since $\Gamma_q$ is abelian, the reduced crossed product is universal). It turns out that the \cst-algebra $\CGD$ is isomorphic to the crossed product $\mathrm{C}_0(\ov{\Gamma}_q)\rtimes_\sigma \Gamma_q$. Furthermore, it is known that $\whG$ is coamenable. Indeed, it was pointed in \cite{Soltanazb,SoltanBohr}. It follows from an easy observation that the universal property of $\mathrm{C}_0(\ov{\Gamma}_q)\rtimes_\sigma \Gamma_q$ together with the trivial representation of $\Gamma_q$ and the character $\mathrm{C}_0(\ov{\Gamma}_q)\ni f \mapsto f(0)\in \CC$ give rise to a character of $\mathrm{C}_0(\ov{\Gamma}_q)\rtimes_\sigma \Gamma_q\simeq \CGD$. Then \cite[Theorem 3.1]{coamenability} implies that $\whG$ is coamenable.\\

One easily checks that the quotient space $\ov{\Gamma}_q/\Gamma_q$ consists of two points and is not antidiscrete. Consequently, \cite[Proposition 7.30]{Williams} implies that $\GG$ is second countable and type I. Using \cite[Theorem 8.39]{Williams} one can describe the spectrum of $\CGD\simeq \mathrm{C}_0(\ov{\Gamma}_q)\rtimes_\sigma \Gamma_q$: there is a family of one dimensional representations indexed by $\widehat{\Gamma}_q$ and one faithful irreducible infinite dimensional representation given by the inclusion into $\B(\LL^2(\Gamma_q))$. Denote this representation by $\pi$. 

\begin{proposition}
The Plancherel measure of $\GG$ equals the Dirac measure at $\pi$, a representation corresponding to the inclusion $\pi\colon \CGD\xrightarrow[]{\simeq} \mathrm{C}_0(\ov{\Gamma}_q)\rtimes_\sigma\Gamma_q\hookrightarrow \B(\LL^2(\Gamma_q))$.
Consequently we have $\mc{Q}_L,\mc{Q}_R\colon \LdG\rightarrow \HS(\LL^2(\Gamma_q))$.
\end{proposition}

\begin{proof}
It is observed in \cite{WoronowiczHaar} that we have $\hpsi\circ \tau^{\whG}_t=|q^{-4it}| \hpsi$ for all $t\in\RR$, hence the scaling constant of $\GG$ equals $\nu={\hat{\nu}}^{-1}=|q^{-4i}|$. 
In the first and the third case $q$ is not real and it follows that $\nu$ is nontrivial. Corollary \ref{wn1} implies that the set of one dimensional representations is of measure zero, and the claim follows\footnote{We remark that it was already observed in \cite{DaeleHaar} that in the first case, $\Linfd$ is isomorphic to the algebra of bounded operators on a separable Hilbert space.}. Let us now consider the second case, i.e.~$q\in \left]0,1\right[$.
It is argued in \cite[Section 5, Proposition A.3]{DaeleHaar} that the von Neumann algebra $\Linfd$ is isomorphic to the von Neumann algebra $\M$ associated with a pair $(a,b)$ of admissible normal operators (see \cite[Definition 5.1]{DaeleHaar}). Moreover, up to an isomorphism $\M$ does not depend on the choice of $(a,b)$, in particular we can take a pair $(a,b)$ introduced in \cite[Proposition 5.2]{DaeleHaar}. In this case one easily sees that the resulting von Neumann algebra equals the whole $\B(\ell^2(\ZZ))$. In particular it is a factor, hence Proposition \ref{stw7} implies that the Plancherel measure of $\GG$ must be the Dirac measure at $\pi$.
\end{proof}

Now we turn to the problem of identifying operators $D_\pi,E_\pi$. To simplify the notation, we will call these operators respectively $D$ and $E$. Let us start with introducing two normal (unbounded) operators on $\LL^2(\Gamma_q)$: $a$ and $b$. Operator $b$ acts by multiplication: $(b f)(\gamma)=\gamma f(\gamma)\,(f\in \Dom(b),\gamma\in\Gamma_q)$ and has the obvious domain. The second operator $a$ can be defined as $a=\mc{F} b \mc{F}^*$.\\
Note that there exists an isomorphism of von Neumann algebras $\Phi_R \colon \Linfd \rightarrow \B(\LL^2(\Gamma_q))$ induced by $\mc{Q}_R J_{\hvp} J_{\vp}$, such that $\Phi_R(x)=\pi(x)$ for $x\in\CGD$ (see Theorem \ref{PlancherelR} and Proposition \ref{stw7}). Under this isomorphism, the right Haar integral $\hpsi$ is transformed to $\Tr(E^{-1} \cdot E^{-1})$ -- it follows from the construction of the Plancherel measure in \cite{Desmedt}. On the other hand, we have $\hpsi(x)=\Tr(|b| \pi(x) |b|)$ for all $x\in \CGD^+$ (\cite[Theorem 3.1]{WoronowiczHaar}). This means that the weights $\Tr(E^{-1}\cdot E^{-1}),\,\Tr(|b| \cdot |b|)$ are equal on $\Phi_R(\CGD)$. Let $\theta$ be the restriction of these weights to $\Phi_R(\CGD)$. The modular automorphism group of $\Tr(E^{-1}\cdot E^{-1})$ is given by $\sigma_t^{\Tr_{E^{-1}}}(A)=E^{-2it} A E^{2it}$, similarly $\sigma^{\Tr_{|b|}}_t(A)=|b|^{2it} A |b|^{-2it}\,(A\in \B(\LL^2(\Gamma_q)),t\in\RR)$. Next, the weight $\theta$ satisfies the KMS condition for both groups $(\sigma_t^{\Tr_{E^{-1}}}|_{\Phi_R(\CGD)})_{t\in\RR}$ and $(\sigma_t^{\Tr_{|b|}}|_{\Phi_R(\CGD)})_{t\in\RR}$ and as this weight is faithful, \cite[Corollary 6.35]{KustermansKMS} implies $E^{-2it} A E^{2it} = |b|^{2it} A |b|^{-2it}$ for all $A\in \Phi_R(\CGD),t\in\RR$. By the $\swot$ density of $\Phi_R(\CGD)$ in $\B(\LL^2(\Gamma_q))$ we get $ E=c |b|^{-1}$ for some $c>0$. Equality $\Tr(E^{-1}\cdot E^{-1})=\Tr( |b|\cdot |b|)$ on $\Phi_R(\CGD)$ forces $c=1$ and consequently $E=|b|^{-1}$.\\
The next step is to identify the operator $D$. Observe that Lemma \ref{lemat19} implies $f(\pi)=1$. Recall (\cite[Section 6.2]{Soltanazb}, \cite[Equation 3.18]{Woronowiczazb}) that operator $a^{-1} \circ b$ is closable and its closure $a^{-1}b$ is normal. Moreover, we have $R^{\whG}(\pi^{-1}(b))=\pi^{-1}(-q a^{-1}b)$. If we combine this information together with Corollary \ref{wn4} and the equality $E=|b|^{-1}$ we arrive at
\[\begin{split}
&\quad\;
\mc{Q}_L^* ( D^{2it} \otimes \I_{\ov{\LL^2(\Gamma_q)}})\mc{Q}_L=
R^{\whG}( \mc{Q}_L^* (
E^{2it}\otimes\I_{\ov{\LL^2(\Gamma_q)}} )\mc{Q}_L)=
R^{\whG}( \mc{Q}_L^* (
|b|^{-2it}\otimes\I_{\ov{\LL^2(\Gamma_q)}} )\mc{Q}_L)\\
&=
\mc{Q}_L^*( |-qa^{-1} b|^{-2it}\otimes \I_{\ov{\LL^2(\Gamma_q)}} )\mc{Q}_L=
\mc{Q}_L^*( |qa^{-1} b|^{-2it}\otimes \I_{\ov{\LL^2(\Gamma_q)}} )\mc{Q}_L,
\end{split}\]
which implies $D=|qa^{-1} b|^{-1}$.

\begin{proposition}
We have $D=|qa^{-1} b|^{-1}$ and $E=|b|^{-1}$.
\end{proposition}

\section{Appendix}

\begin{lemma}\label{lemat2}
Let $\msf{H}$ be a Hilbert space and $J\colon \msf{H}\otimes \ov{\msf{H}}\rightarrow \msf{H}\otimes\ov{\msf{H}}$ an antilinear map given by $J\colon \xi\otimes\ov\eta\mapsto\eta\otimes\ov\xi\;(\xi,\eta\in\msf{H})$.
\begin{enumerate}[label=\arabic*)]
\item If $x,y\in\B(\msf{H})$ are unitaries such that $x\otimes(y^*)^{\msf T}=y\otimes (x^*)^{\msf T}\in\B(\msf{H}\otimes\ov{\msf{H}})$, then operators $xy^*,y^*x$ are selfadjoint.
\item Let $(a_t)_{t\in \RR},(b_t)_{t\in\RR}$ be families of unitary operators on $\msf{H}$. Define $c_t=a_t\otimes b_t^{\msf T}\,(t\in\RR)$. If $(b_t)_{t\in\RR}$ and $(c_t)_{t\in\RR}$ are strongly continuous groups, then $(a_t)_{t\in\RR}$ is also a strongly continuous group.
\item Let $(a_t)_{t\in \RR},(b_t)_{t\in\RR}$ be strongly continuous groups of unitary operators on $\msf{H}$. Define $c_t=a_t\otimes b_t^{\msf T}\,(t\in\RR)$. If $J c_t=c_t J$ for all $t\in\RR$ then $a_t=b_{-t}\,(t\in\RR)$.
\end{enumerate}
\end{lemma}

\begin{proof}
$1)$ Equality from the assumption gives us $x S y^*=ySx^*$ for all $S\in\HS(\msf{H})$. We can approximate the unit by Hilbert-Schmidt operators hence $xy^*=yx^*$, i.e.~$xy^*$ is selfadjoint. Multiplying this equation from the left by $x^*$ and from the right by $x$ gives us $y^*x=x^*y$, i.e.~$y^*x$ is selfadjoint.\\
$2)$
For all $t,s\in\RR$ we have $
a_{t+s}\otimes b_{t+s}^{\msf T}=c_{t+s}=c_t c_s=a_t a_s \otimes b_t^{\msf T} b_s^{\msf T}$ hence $a_{t+s}=a_ta_s$, i.e.~$(a_t)_{t\in\RR}$ is a group. Equation $a_t\otimes\I_{\ov{\msf{H}}}=c_t (\I_{\msf{H}}\otimes b_{-t}^{\msf T})$ implies that $\RR\ni t\mapsto a_t\in\B(\msf{H})$ is strongly continuous.\\
$3)$ We have $Jc_tJ=b_{-t}\otimes a_{-t}^{\msf T}$ for all $t\in\RR$. Consequently, for  $s,t\in\RR,x\in\B(\msf{H})$ we have
\[\begin{split}
&\quad\;
c_s J c_t J (x\otimes\I_{\ov{\msf{H}}}) J c_{-t} J c_{-s}=
(a_s b_{-t}\otimes b_s^{\msf T} a_{-t}^{\msf T})(x\otimes\I_{\ov{\msf{H}}})
(b_t a_{-s}\otimes a_{t}^{\msf T} b_{-s}^{\msf T})= 
a_s b_{-t} x b_t a_{-s}\otimes \I_{\ov{\msf{H}}},
\end{split}\]
and on the other hand
\[\begin{split}
&\quad\;
J c_t J c_s (x\otimes \I_{\ov{\msf{H}}}) c_{-s} J c_{-t} J=
(b_{-t} a_s\otimes a_{-t}^{\msf T} b_s^{\msf T})(x\otimes\I_{\ov{\msf{H}}})
(a_{-s} b_t\otimes b_{-s}^{\msf T}a_t^{\msf T})=
b_{-t} a_s x a_{-s} b_t \otimes \I_{\ov{\msf{H}}},
\end{split}\]
hence
\[
a_s b_{-t} x b_t a_{-s}=b_{-t} a_s x a_{-s} b_t\quad\quad\Rightarrow\quad\quad
a_{-s} b_t a_s b_{-t} x =
xa_{-s} b_t a_s b_{-t}.
\]
The above equation holds for all $x\in\B(\msf{H})$, hence there exists $\lambda_{t,s}\in \CC$ such that $a_{-s} b_t a_s b_{-t} =\lambda_{t,s}\I_{\msf{H}}$ and consequently $
a_{-s} b_t =\lambda_{t,s} b_t a_{-s}\,(t,s\in\RR)$.
Clearly we have $|\lambda_{t,s}|=1$. Since $
a_t\otimes b_t^{\msf T}=c_t = JJc_t=Jc_t J=b_{-t}\otimes (a_{-t})^{\msf T},
$
for all $t\in\RR$, the first point implies that $a_t b_t,b_ta_t$ are selfadjoint. For $s=-t$ we have $a_{t} b_t=\lambda_{t,-t} b_t a_t$, and since $a_tb_t,b_ta_t$ are selfadjoint we have $\lambda_{t,-t}\in\RR\cap \TT=\{-1,1\}$. As the function $t\mapsto \lambda_{t,-t}\in\RR$ is continuous and $\lambda_{0,0}=1$, we have $\lambda_{t,-t}=1$ for all $t\in\RR$. Consequently $
b_ta_t=a_tb_t=(a_tb_t)^*=b_{-t} a_{-t}$
and $b_{2t}=a_{-2t}\quad(t\in\RR).$
\end{proof}

\section*{Acknowledgements}
The author would like to express his gratitude towards Piotr M.~Sołtan for many helpful discussions and suggestions.\\
The author was partially supported by the Polish National Agency for
the Academic Exchange, Polonium grant PPN/BIL/2018/1/00197, FWO–PAS
project VS02619N: von Neumann algebras arising from quantum symmetries and NCN (National Centre of Science) grant 2014/14/E/ST1/00525.

\bibliographystyle{plain}
\bibliography{bibliografia}

\begin{thebibliography}{10}

\bibitem{bmt}
E.~B{\'e}dos, G.~J. Murphy, and L.~Tuset.
\newblock Co-amenability of compact quantum groups.
\newblock {\em J. Geom. Phys.}, 40(2):130--153, 2001.

\bibitem{coamenability}
E.~B\'{e}dos and L.~Tuset.
\newblock Amenability and co-amenability for locally compact quantum groups.
\newblock {\em Internat. J. Math.}, 14(8):865--884, 2003.

\bibitem{Caspers}
M.~Caspers.
\newblock {\em Non-commutative integration on locally compact quantum groups:
  Fourier theory, Gelfand pairs, non-commutative $L^p$-spaces}.
\newblock PhD thesis, Radboud Universiteit Nijmegen, 2012.

\bibitem{CaspersKoelink}
M.~Caspers and E.~Koelink.
\newblock Modular properties of matrix coefficients of corepresentations of a
  locally compact quantum group.
\newblock {\em J. Lie Theory}, 21(4):905--928, 2011.

\bibitem{Desmedt}
P.~Desmedt.
\newblock {\em Aspects of the theory of locally compact quantum groups:
  Amenability - Plancherel measure}.
\newblock PhD thesis, Katholieke Universiteit Leuven, 2003.

\bibitem{DixmiervNA}
J.~Dixmier.
\newblock {\em {V}on {N}eumann algebras}, volume~27 of {\em North-Holland
  Mathematical Library}.
\newblock North-Holland Publishing Co., Amsterdam-New York, 1981.

\bibitem{Krajczok}
J.~{Krajczok}.
\newblock {Coamenability of type I locally compact quantum groups}.
\newblock {\em arXiv e-prints}, page arXiv:2001.06740, 2020.

\bibitem{quantumdisk}
J.~{Krajczok} and P.~M. {So{\l}tan}.
\newblock {The quantum disk is not a quantum group}.
\newblock {\em arXiv e-prints}, page arXiv:2005.02967, 2020.

\bibitem{KustermansKMS}
J.~Kustermans.
\newblock {KMS}-weights on $\mathrm{C}^*$-algebras, 1997.

\bibitem{Kustermans}
J.~Kustermans.
\newblock Locally compact quantum groups in the universal setting.
\newblock {\em Internat. J. Math.}, 12(3):289--338, 2001.

\bibitem{KustermansVaes}
J.~Kustermans and S.~Vaes.
\newblock Locally compact quantum groups in the von {N}eumann algebraic
  setting.
\newblock {\em Math. Scand.}, 92(1):68--92, 2003.

\bibitem{Lance}
C.~Lance.
\newblock Direct integrals of left {H}ilbert algebras.
\newblock {\em Math. Ann.}, 216:11--28, 1975.

\bibitem{NeshTu}
S.~Neshveyev and L.~Tuset.
\newblock {\em Compact quantum groups and their representation categories},
  volume~20 of {\em Cours Sp\'ecialis\'es [Specialized Courses]}.
\newblock Soci\'et\'e Math\'ematique de France, Paris, 2013.

\bibitem{Soltanazb}
P.~M. So{\l}tan.
\newblock New quantum ``{$az+b$}'' groups.
\newblock {\em Rev. Math. Phys.}, 17(3):313--364, 2005.

\bibitem{SoltanBohr}
P.~M. So{\l}tan.
\newblock Quantum {B}ohr compactification.
\newblock {\em Illinois J. Math.}, 49(4):1245--1270, 2005.

\bibitem{MUQGII}
P.~M. So{\l}tan and S.~L. Woronowicz.
\newblock From multiplicative unitaries to quantum groups. {II}.
\newblock {\em J. Funct. Anal.}, 252(1):42--67, 2007.

\bibitem{TakesakiII}
M.~Takesaki.
\newblock {\em Theory of operator algebras. {II}}, volume 125 of {\em
  Encyclopaedia of Mathematical Sciences}.
\newblock Springer-Verlag, Berlin, 2003.
\newblock Operator Algebras and Non-commutative Geometry, 6.

\bibitem{Vaksman}
L.~L. Vaksman and Ya.~S. Soibelman.
\newblock An algebra of functions on the quantum group {${\rm SU}(2)$}.
\newblock {\em Funktsional. Anal. i Prilozhen.}, 22(3):1--14, 96, 1988.

\bibitem{DaeleHaar}
A.~van Daele.
\newblock The {H}aar measure on some locally compact quantum groups, 2001.

\bibitem{Daele}
A.~van Daele.
\newblock Locally compact quantum groups. {A} von {N}eumann algebra approach.
\newblock {\em SIGMA Symmetry Integrability Geom. Methods Appl.}, 10:Paper 082,
  41, 2014.

\bibitem{Williams}
D.~P. Williams.
\newblock {\em Crossed products of {$C{^\ast}$}-algebras}, volume 134 of {\em
  Mathematical Surveys and Monographs}.
\newblock American Mathematical Society, Providence, RI, 2007.

\bibitem{Woronowiczsu2}
S.~L. Woronowicz.
\newblock Twisted {SU}(2) group. {A}n example of a noncommutative differential
  calculus.
\newblock {\em Publ. Res. Inst. Math. Sci.}, 23(1):117--181, 1987.

\bibitem{cqg}
S.~L. Woronowicz.
\newblock Compact quantum groups.
\newblock In {\em Sym\'etries quantiques ({L}es {H}ouches, 1995)}, pages
  845--884. North-Holland, Amsterdam, 1998.

\bibitem{QEF}
S.~L. Woronowicz.
\newblock Quantum exponential function.
\newblock {\em Rev. Math. Phys.}, 12(6):873--920, 2000.

\bibitem{Woronowiczazb}
S.~L. Woronowicz.
\newblock Quantum ``{$az+b$}'' group on complex plane.
\newblock {\em Internat. J. Math.}, 12(4):461--503, 2001.

\bibitem{WoronowiczHaar}
SL~Woronowicz.
\newblock {Haar weight on some quantum groups}.
\newblock In JP~Gazeau, R~Kerner, JP~Antoine, S~Metens, and JY~Thibon, editors,
  {\em {Group 24 : Physical And Mathematical Aspects Of Symmetries}}, volume
  {173} of {\em {Institute Of Physics Conference Series}}, pages {763--772},
  {2003}.

\end{thebibliography}
\end{document}